\documentclass{amsart}

\usepackage{amsmath}
\usepackage{amsfonts}
\usepackage{amssymb}
\usepackage{pdfsync}
\usepackage{color,graphicx}
\usepackage{tikz}
\usepackage{a4wide}
\usepackage{amscd,amsthm,latexsym}
\usepackage{hyperref}

\newtheorem{thm}{Theorem}[section]
\newtheorem{lem}[thm]{Lemma}
\newtheorem{prop}[thm]{Proposition}
\newtheorem{cor}[thm]{Corollary}
\theoremstyle{definition}
\newtheorem{algorithm}[thm]{Algorithm}

\newtheorem{defn}[thm]{Definition}

\newtheorem{remark}[thm]{Remark}

\newtheorem{problem}[thm]{Problem}

\newcommand\ZZ{\mathbb{Z}}

\newcommand\CT{{\operatorname{CT}}}
\newcommand\HT{{\operatorname{HT}}}
\newcommand\SSCT{{\operatorname{SSCT}}}
\newcommand\SYT{{\operatorname{SYT}}}

\newcommand\ST{{\operatorname{ST}}}

\newcommand\vjdt{\phi_{\operatorname{vjdt}}}
\newcommand\mjdt{\phi_{\operatorname{mjdt}}}
\newcommand\vsort{\phi_{\operatorname{vsort}}}
\newcommand\msort{\phi_{\operatorname{msort}}}
\newcommand\tail{{\operatorname{tail}}}
\newcommand\head{{\operatorname{head}}}

\newcommand\LHT{{\operatorname{LHT}}}

\newcommand\SSYT{{\operatorname{SSYT}}}

\newcommand\leg{\operatorname{leg}}
\newcommand\arm{\operatorname{arm}}

\newcommand\Sym{\operatorname{Sym}}
\newcommand\lm{{\lambda/\mu}}
\newcommand\wt{\operatorname{wt}}

\renewcommand\vec[1]{\mathbf{#1}}

\newcommand\flr[1]{\left\lfloor #1\right\rfloor}

\newcommand\NN{\mathbb{N}}
\newcommand\GG{\mathcal{G}}

\newcommand\cell[3]{
\def\i{#1} \def\j{#2} \def\entry{#3}
\draw (\j-1,-\i)--(\j,-\i)--(\j,-\i+1);
\node at (\j-.5,-\i+.5) {\entry};
}
\newcommand\dcell[3]{
\def\i{#1} \def\j{#2} \def\entry{#3}
\draw (2*\j-2,-\i)--(2*\j,-\i)--(2*\j,-\i+1);
\node at (2*\j-1,-\i+.5) {\entry};
}

\newcommand\LHLL[2]{
  \def\X{#1} \def\Y{#2}
  \foreach \i in {0,...,\X}
  {
\pgfmathsetmacro{\m}{\Y+.5};
    \draw[gray,very thin] (\i,0) -- (\i,\m);
    \node at (\i,-.3) {\i};
  }
\pgfmathsetmacro{\m}{\Y+.5};
\node at (-.3,\m) {$\infty$};
\node at (2.5,\Y+0.5) {$\vdots$};\node at (6.5,\Y+0.5) {$\vdots$};
\node at (-.3,\Y) {$\vdots$};
\pgfmathsetmacro{\m}{\Y-1};
  \foreach \j in {0,...,\m}
  {
    \node at (-.3,\j) {\j};
  }
\pgfmathsetmacro{\m}{\Y};
  \foreach \x in {1,...,\X}
{ \foreach \j in {0,...,\m}
  {
\draw[gray,very thin] (\x-1,\j) -- (\x,\j);
}}
\pgfmathsetmacro{\m}{\Y-1};
 \foreach \i in {2,...,\X}
 \foreach \j in {0,...,\m} 
{\foreach \y in {2,...,\i}
{
\pgfmathsetmacro{\w}{\j+(\y-1)/\i};
\pgfmathsetmacro{\z}{\j+(\y-1)/(\i+1)};
\draw[gray,very thin] (\i-1,\w) -- (\i,\z);
}}
\node at (\X+1,1) {$\cdots$};
\node at (\X+1,\Y-1) {$\cdots$};
}

\newcommand\ELHLL[2]{
  \def\X{#1} \def\Y{#2}
  \foreach \i in {0,...,\X}
  {
\pgfmathsetmacro{\m}{\Y+1.5};
    \draw[gray,very thin] (\i,0) -- (\i,\m);
    \node at (\i,-.3) {\i};
  }
\pgfmathsetmacro{\m}{\Y+.5};
\node at (-.3,\m) {$\omega$};
\node at (-.6,\m+1) {$\omega+1$};
\node at (2.5,\Y) {$\vdots$};
\node at (6.5,\Y) {$\vdots$};
\node at (-.3,\Y) {$\vdots$};
\pgfmathsetmacro{\m}{\Y-1};
  \foreach \j in {0,...,\m}
  {
    \node at (-.3,\j) {\j};
  }
\pgfmathsetmacro{\m}{\Y-1};
  \foreach \x in {1,...,\X}
{ \foreach \j in {0,...,\m}
  {
\draw[gray,very thin] (\x-1,\j) -- (\x,\j);
}}
\pgfmathsetmacro{\m}{\Y-2};
 \foreach \i in {2,...,\X}
 \foreach \j in {0,...,\m} 
{\foreach \y in {2,...,\i}
{
\pgfmathsetmacro{\w}{\j+(\y-1)/\i};
\pgfmathsetmacro{\z}{\j+(\y-1)/(\i+1)};
\draw[gray,very thin] (\i-1,\w) -- (\i,\z);
}}
\pgfmathsetmacro{\j}{\Y+.5};
 \foreach \i in {1,...,\X}
{\foreach \y in {1,...,\i}
{
\pgfmathsetmacro{\w}{\j+(\y-1)/\i};
\pgfmathsetmacro{\z}{\j+(\y-1)/(\i+1)};
\draw[gray,very thin] (\i-1,\w) -- (\i,\z);
}}
\node at (\X+1,1) {$\cdots$};
\node at (\X+1,\Y-1) {$\cdots$};
\node at (\X+1,\Y+1) {$\cdots$};
}

\newcommand\wLHLL[1]{
  \def\X{#1} \def\Y{1}
  \foreach \i in {0,...,\X}
  {
    \draw[gray,very thin] (\i,0) -- (\i,\Y);
    \node at (\i,-.3) {\i};
  }
\pgfmathsetmacro{\m}{\Y+.5};
\node at (-.3,0) {$\omega$};
\node at (-.6,1) {$\omega+1$};
\pgfmathsetmacro{\j}{0};
 \foreach \i in {1,...,\X}
{\foreach \y in {1,...,\i}
{
\pgfmathsetmacro{\w}{\j+(\y-1)/\i};
\pgfmathsetmacro{\z}{\j+(\y-1)/(\i+1)};
\draw[gray,very thin] (\i-1,\w) -- (\i,\z);
}}
\node at (\X+1,.5) {$\cdots$};
}

\title{Enumeration of bounded lecture hall tableaux}

\author{Sylvie Corteel and Jang Soo Kim}
\address{University of California, Berkeley, United States}
\email{corteel@berkeley.edu}
\address{Sungkyunkwan University, Suwon,  South Korea}
\email{jangsookim@skku.edu}

\date{\today}



\begin{document}
\maketitle
\begin{center}
{\em To Christian Krattenthaler, our determinantal hero.}
\end{center}

\begin{abstract}
Recently the authors introduced lecture hall tableaux in their study of multivariate little $q$-Jacobi polynomials. In this paper, we enumerate bounded lecture hall tableaux. We show that their enumeration is closely related to standard and semistandard Young tableaux. We also show that the number of bounded lecture hall tableaux is the coefficient of the Schur expansion 
of $s_\lambda(m+y_1,\dots,m+y_n)$. To prove this result, we use two main tools: non-intersecting lattice paths and bijections. In particular we use ideas developed by Krattenthaler to prove bijectively the hook content formula.

\end{abstract}


\section{Introduction}

Recently the authors \cite{LHT} introduced lecture hall tableaux in their study of multivariate little $q$-Jacobi polynomials $P_\lambda(x;a,b;q)$ with $t=q$. 
They showed that if we expand the Schur function $s_\lambda(x)$ in terms of $P_\mu(x;a,b;q)$
and vice versa as
\[
s_\lambda(x) = \sum_{\mu} M_{\lambda,\mu} P_\mu(x;a,b;q),\qquad
P_\lambda(x;a,b;q) = \sum_{\mu} N_{\lambda,\mu}s_\mu(x),
\]
then the coefficients $M_{\lambda,\mu}$ and $N_{\lambda,\mu}$ can be expressed as generating functions for lecture hall tableaux of shape $\lm$.

A lecture hall tableau is a certain filling of a skew shape $\lm$ with nonnegative integers. Since the entries in a lecture hall tableau can be arbitrarily large, there are infinitely many lecture hall tableaux of a given shape. If we give an upper bound on their entries we can consider the number of lecture hall tableaux.  The main goal of this paper is to enumerate such bounded lecture hall tableaux. 

Bounded lecture hall objects were first enumerated by the first author, Lee and Savage in \cite{CLS}. They showed that the number of sequences $\lambda=(\lambda_1,\ldots ,\lambda_n)$ of integers such that
$$
m\ge \frac{\lambda_1}{1}\ge \frac{\lambda_2}{2}\ge \dots \ge   \frac{\lambda_n}{n}\ge 0
$$
is equal to the number of sequences $\lambda=(\lambda_1,\ldots ,\lambda_n)$ of integers such that
$$
m\ge \frac{\lambda_1}{n}\ge \frac{\lambda_2}{n-1}\ge \dots \ge   \frac{\lambda_n}{1}\ge 0.
$$
This number is equal to $(m+1)^n$.  As remarked by Matt Beck \cite{Matt}, 
this is also the  Ehrhart polynomial of the $n$-cube.
This observation started a collection of very interesting papers connecting lecture hall partitions to geometric
combinatorics and in particular polytopes. We cite for example \cite{BBKSZ1,BBKSZ2,liustanley,Olsen}.
An overview of the techniques and results is presented in the survey by Carla Savage \cite{LHPSavage}.

We will see that counting bounded lecture hall tableaux is naturally related to standard and semistandard Young tableaux.  To state our results we first give definitions of related objects.

A \emph{partition} is a weakly decreasing sequence $\lambda=(\lambda_1,\dots,\lambda_k)$ of positive integers. Each integer $\lambda_i$ is called a \emph{part} of $\lambda$. The \emph{length} $\ell(\lambda)$ of $\lambda$ is the number of parts. We identify a partition $\lambda=(\lambda_1,\dots,\lambda_k)$ with its \emph{Young diagram}, which is a left-justified array of squares, called \emph{cells}, with $\lambda_i$ cells in the $i$th row for $1\le i\le k$.
In other words, we consider $\lambda=(\lambda_1,\dots,\lambda_k)$ as the set of cells $(i,j)$ such that $1\le i\le k$ and $1\le j\le \lambda_i$. 
For two partitions $\lambda$ and $\mu$ we write $\mu\subset \lambda$ to mean that the Young diagram of $\mu$ is contained in that of $\lambda$ as a set. In this case, a \emph{skew shape} $\lm$ is defined to be the set-theoretic difference $\lambda\setminus \mu$ of their Young diagrams. 
We denote by $|\lm|$ the number of cells in $\lm$. A partition $\lambda$ is also considered as a skew shape by $\lambda=\lambda/\emptyset$. 

A \emph{tableau} of shape $\lm$ is a filling of the cells in $\lm$ with nonnegative integers. In other words, a tableau is a map $T:\lm\to \NN$, where $\NN$ is the set of nonnegative integers. A \emph{standard Young tableau} of shape $\lm$ is a tableau of shape $\lm$ such that every integer $1\le i\le |\lm|$ appears exactly once and the entries are decreasing in each row and in each column. Let $\SYT(\lm)$ denote the set of standard Young tableaux of shape $\lm$. We note that it is more common to define a standard Young tableau to have entries increasing in each row and column. However, for our purpose in this paper, it is more convenient to have entries decreasing.

It is well known that the number of standard Young tableaux of shape $\lambda$ is given by the \emph{hook length formula} due to Frame, Robinson, and Thrall \cite{Frame1954}:
\begin{equation}
  \label{eq:6}
|\SYT(\lambda)| = \frac{|\lambda|!}{\prod_{(i,j)\in\lambda} h(i,j)},  
\end{equation}
where $h(i,j)=\lambda_i+\lambda_j'-i-j+1$ and $\lambda_j'$ is the number of integers $1\le r\le \ell(\lambda)$ with $\lambda_r\ge j$. 
There are many proofs of the hook length formula, see the survey by Adin and Roichman \cite{Adin2015}.
Among these a remarkable bijective proof of \eqref{eq:6} was found by Novelli, Pak, and Stoyanovskii \cite{NPS} using a ``jeu de taquin'' sorting algorithm. 

A \emph{semistandard Young tableau} of shape $\lm$ is a tableau of $\lm$ such that the entries are weakly decreasing in each row and strictly decreasing in each column. 
We denote by $\SSYT(\lm)$ the set of semistandard Young tableaux of shape $\lm$. 
We also denote by $\SSYT_n(\lm)$ the set of $T\in\SSYT(\lm)$ with $\max(T)<n$, i.e., the entries of $T$ are taken from $\{0,1,\dots,n-1\}$. 
Stanley \cite{Stanley1971} showed that the number of such bounded semistandard Young tableaux is given by the \emph{hook-content formula}:
\begin{equation}
  \label{eq:7}
|\SSYT_n(\lambda)|= \prod_{(i,j)\in\lambda} \frac{n+c(i,j)}{h(i,j)},
\end{equation}
where $c(i,j)=j-i$ is the \emph{content} of the cell $(i,j)$.  There are also many proofs of the hook content formula. 
Krattenthaler \cite{Kratt1999} found a bijective proof of \eqref{eq:7} that uses a modified jeu de taquin sorting algorithm. 
In this paper we will use Krattenthaler's jeu de taquin to investigate lecture hall tableaux. 

An \emph{$n$-lecture hall tableau} of shape $\lm$ is a tableau $L$ of shape $\lm$ satisfying the following conditions:
\[
\frac{L(i,j)}{n+c(i,j)}  \ge \frac{L(i,j+1)}{n+c(i,j+1)}, \qquad
\frac{L(i,j)}{n+c(i,j)} > \frac{L(i+1,j)}{n+c(i+1,j)}.
\]
The set of $n$-lecture hall tableaux of shape $\lm$ is
denoted by $\LHT_n(\lm)$. For $L\in \LHT_{n}(\lm)$, let $\flr{L}$ be the tableau of shape $\lm$ whose
$(i,j)$-entry is $\flr{L(i,j)/(n-i+j)}$, see Figure~\ref{fig:LHT1} for an example. 
 The set of $n$-lecture hall tableaux $L\in \LHT_{n}(\lm)$ with $\max(\flr{L})<m$ is denoted by $\LHT_{n,m}(\lm)$. 
Since the bounded lecture hall tableaux in $\LHT_{n,1}(\lm)$ play an important role in our paper, 
we give a special name for them. These objects have another description as follows. 

\begin{figure}
  \centering
\begin{tikzpicture}[scale=.6]
\cell14{25} \cell15{25} \cell16{21}
\cell22{16} \cell23{18} \cell24{21} \cell25{10} \cell264
\cell318 \cell329 \cell332 \cell340
\cell414 \cell424 \cell430
\draw (0,-4)--(0,-2)--(1,-2)--(1,-1)--(3,-1)--(3,0)--(6,0);
\end{tikzpicture} \qquad \qquad
\begin{tikzpicture}[scale=.6]
\cell14{$\frac{25}8$} \cell15{$\frac{25}9$} \cell16{$\frac{21}{10}$}
\cell22{$\frac{16}5$} \cell23{$\frac{18}6$} \cell24{$\frac{21}7$} \cell25{$\frac{10}8$} \cell26{$\frac49$}
\cell31{$\frac83$} \cell32{$\frac94$} \cell33{$\frac25$} \cell34{$\frac06$}
\cell41{$\frac42$} \cell42{$\frac43$} \cell43{$\frac04$}
\draw (0,-4)--(0,-2)--(1,-2)--(1,-1)--(3,-1)--(3,0)--(6,0);
\end{tikzpicture}\qquad \qquad
\begin{tikzpicture}[scale=.6]
\cell14{$3$} \cell15{$2$} \cell16{$2$}
\cell22{$3$} \cell23{$3$} \cell24{$3$} \cell25{$1$} \cell26{$0$}
\cell31{$2$} \cell32{$2$} \cell33{$0$} \cell34{$0$}
\cell41{$2$} \cell42{$1$} \cell43{$0$}
\draw (0,-4)--(0,-2)--(1,-2)--(1,-1)--(3,-1)--(3,0)--(6,0);
\end{tikzpicture}
  \caption{On the left is a lecture hall tableau $L\in \LHT_n(\lm)$ for $n=5$, $\lambda=(6,6,4,3)$ and $\mu=(3,1)$. 
The diagram in the middle shows the number $L(i,j)/(n+c(i,j))$ for each entry $(i,j)\in\lm$. The diagram on the right is the tableau $\flr{L}$.}
  \label{fig:LHT1}
\end{figure}

A \emph{semistandard $n$-content tableau} of shape $\lm$ is a semistandard Young tableau $S$ of shape $\lm$ with the additional condition that $0\le S(i,j)<n-i+j$ for every $(i,j)\in\lm$. We denote by $\SSCT_n(\lm)$ the set of semistandard $n$-content tableaux of shape $\lm$.  It is easy to see that  
\begin{align*}
\SSCT_n(\lm)&=\LHT_{n,1}(\lm),\\
\SSCT_n(\lambda)&=\SSYT_n(\lambda).  
\end{align*}

In this paper we prove the following formula for the number of bounded lecture hall tableaux. 
Given a partition $\mu$, we use the convention that $\mu_i=0$ for all integers $i>\ell(\mu)$.

\begin{thm}\label{thm:1}
For partitions $\lambda$ and $\mu$ with $\mu\subset \lambda$ and $\ell=\ell(\lambda)\le n$, we have
\[
|\LHT_{n,m}(\lm)|  = m^{|\lm|} \det\left(\binom{\lambda_i+n-i}{\mu_j+n-j}\right)_{1\le i,j\le \ell}.
\]
\end{thm}

Note that Theorem~\ref{thm:1} implies that
\begin{equation}
  \label{eq:LHT-SSCT}
|\LHT_{n,m}(\lm)|  = m^{|\lm|}|\LHT_{n,1}(\lm)|  = m^{|\lm|}|\SSCT_{n}(\lm)| .  
\end{equation}
The determinant in Theorem~\ref{thm:1} has another description in terms of standard Young tableaux. 

\begin{prop}\label{prop:KS}
For partitions $\lambda$ and $\mu$ with $\mu\subset \lambda$ and $\ell=\ell(\lambda)\le n$, we have
\[
|\SSCT_n(\lm)|=\det\left(\binom{\lambda_i+n-i}{\mu_j+n-j}\right)_{1\le i,j\le \ell}
=\frac{|\SYT(\lm)|}{|\lm|!} \prod_{x\in\lm} (n+c(x)).
\]
\end{prop}

Kirillov and Scrimshaw \cite{KS} recently conjectured that the number $\frac{|\SYT(\lm)|}{|\lm|!} \prod_{x\in\lm} (n+c(x))$ on the right hand side of the identity in Proposition~\ref{prop:KS} is always an integer and proposed a problem to find a combinatorial object for this number. Proposition~\ref{prop:KS} gives an affirmative answer to the problem.  Theorem~\ref{thm:1} and Proposition~\ref{prop:KS} together with \eqref{eq:6} and \eqref{eq:7} immediately imply the following corollary. 

\begin{cor}
For partitions $\lambda$ and $\mu$ with $\mu\subset \lambda$ and $\ell(\lambda)\le n$, we have
\[
|\LHT_{n,m}(\lm)| = m^{|\lm|} \frac{|\SYT(\lm)|}{|\lm|!} \prod_{x\in\lm} (n+c(x)).
\]
In particular, the number of $n$-lecture hall tableaux of shape $\lambda$ whose maximum entry is less than $nm$ is 
\[
|\LHT_{n,m}(\lambda)| = m^{|\lambda|} |\SSYT_n(\lambda)| = m^{|\lambda|} \prod_{x\in\lambda} \frac{n+c(x)}{h(x)}.
\]
\end{cor}

Using Naruse's hook length formula for $|\SYT(\lm)|$ in \cite{Naruse}, we get another enumerative formula:
\begin{cor}\label{cor:naruse}
For partitions $\lambda$ and $\mu$ with $\mu\subset \lambda$ and $\ell(\lambda)\le n$, the number of bounded lecture tableaux of shape $\lm$ is
$$
|\LHT_{n,m}(\lm)|=m^{|\lm|}  \prod_{x\in\lm} (n+c(x))\sum_D \prod_{x\in \lambda\backslash D} \frac{1}{h(x)},
$$
where the sum is over all excited diagrams $D$ of $\lm$. See \cite{MPP,Naruse} for details on excited diagrams.
\end{cor}

In this paper we also show that the number of bounded lecture hall tableaux occurs naturally as the coefficient in the Schur expansion of $s_\lambda(m+y_1,\dots,m+y_n)$. 
Recall that for a sequence of variables $\vec x=(x_0,x_1,\dots)$, the \emph{(skew) Schur function} $s_\lm(\vec x)$ is defined by
\[
s_\lm(\vec x) = \sum_{T\in\SSYT(\lm)} \vec x^T,
\]
where $\vec x^T=\prod_{(i,j)\in\lm} x_{T(i,j)}$. Note that
\[
s_\lambda(x_0,x_1,\dots,x_{n-1})=\sum_{T\in\SSYT_n(\lambda)} \vec x^T,
\]
and $|\SSYT_n(\lambda)|=s_\lambda(1^n)$, where $(1^n)$ is the sequence $(1,1,\dots,1)$ of $n$ ones.

\begin{thm}\label{thm:2}
For integers $n,m\ge0$, variables $y_1,\dots,y_n$, and a partition $\lambda$ with at most $n$ parts, we have
\[
s_\lambda(m+y_1,\dots,m+y_n)=\sum_{\mu\subset\lambda} |\LHT_{n,m}(\lm)| s_\mu(y_1,\dots,y_n).
\]  
\end{thm}

If $m=1$ in Theorem~\ref{thm:2} we obtain the following formula due to Lascoux \cite{Lascoux1978}:
\begin{equation}
  \label{eq:Lascoux}
s_\lambda(1+y_1,\dots,1+y_n)=\sum_{\mu\subset\lambda} \det\left(\binom{\lambda_i+n-i}{\mu_j+n-j}\right)_{1\le i,j\le \ell(\lambda)} s_\mu(y_1,\dots,y_n).
\end{equation}
Lascoux \cite{Lascoux1978} used \eqref{eq:Lascoux} to compute the Chern classes of the exterior square and symmetric square of a vector bundle,
see  also \cite[Chapter 1, \S3, Example 10]{Macdonald}. We note Theorem~\ref{thm:2} can also be obtained from \eqref{eq:Lascoux} and Theorem~\ref{thm:1}. 

Our next theorem is a generalization of Theorem~\ref{thm:2} to skew shapes. In order to state the theorem we first need to introduce some definitions. 

For any tableau $T$ of shape $\lm$, let
\[
\vec x^T = \prod_{(i,j)\in\lm} x_{T(i,j)}.
\]
We define
\[
L_{\lm}^{n}(\vec x) = \sum_{T\in\LHT_{n}(\lm)} \vec x^{\flr{T}},
\]
and
\[
S_{\lm}^{n}(\vec x) = \sum_{T\in\SSCT_{n}(\lm)} \vec x^{T}.
\]
Note that $S_{\lambda}^{n}(\vec x)= s_\lambda(x_0,x_1,\dots,x_{n-1})$. 

The following theorem is the main theorem of this paper, which is a skew version of Theorem~\ref{thm:2}.

\begin{thm}\label{thm:main}
Let $\lambda$ and $\mu$ be partitions with $\mu\subset \lambda$ and $\ell(\lambda)\le n$.
For any sequences $\vec x=(x_0,x_1,\dots)$ and $\vec y=(y_0,y_1,\dots)$ of variables, we have
\[
S^n_{\lm}(|\vec x|+\vec y) = \sum_{\mu\subset\nu\subset\lambda} L_{\lambda/\nu}^{n}(\vec x) S^n_{\nu/\mu}(\vec y),
\]  
where $|\vec x|=x_0+x_1+\cdots$ and $|\vec x|+\vec y = (|\vec x|+y_0,|\vec x|+y_1,\dots)$. 
\end{thm}
In this paper we give two proofs of Theorem~\ref{thm:main}: one proof uses a Jacobi--Trudi type determinant identity and the other proof is bijective. In particular the bijective proof of Theorem~\ref{thm:main} uses a variation of jeu de taquin due to Krattenthaler \cite{Kratt1999}. 

If $\mu=\emptyset$ and $\vec x=(1^m)$, 
in Theorem~\ref{thm:main},  we have
\[
S^n_{\lambda}(m+\vec y) = \sum_{\nu\subset\lambda} L_{\lambda/\nu}^{n}(1^m) S^n_{\nu}(\vec y).
\]  
Since $S^n_{\nu}(\vec y) = s_{\nu}(y_0,y_1,\dots,y_{n-1})$ for any partition $\nu$, we obtain Theorem~\ref{thm:2}. 

We can also deduce \eqref{eq:LHT-SSCT} from Theorem~\ref{thm:main} as follows. 
If $\vec x=(x_0,\ldots ,x_{m-1})$ and $\vec y = (0,0,\dots)$ in Theorem~\ref{thm:main},  we have
\begin{equation}
  \label{eq:3}
S^n_{\lm}(|\vec x|,|\vec x|,\dots) = L_{\lambda/\mu}^{n}(\vec x) .  
\end{equation}
By definition we have $L_{\lambda/\mu}^{n}(1^m) =|\LHT_{n,m}(\lm)|$ and 
\begin{equation}
  \label{eq:4}
S^n_{\lm}(|\vec x|,|\vec x|,\dots)= |\vec x|^{|\lm|}S^n_{\lm}(1,1,\dots) = |\vec x|^{|\lm|} |\SSCT_n(\lm)|.
\end{equation}
Then \eqref{eq:LHT-SSCT} follows from \eqref{eq:3}, \eqref{eq:4} with $\vec x = (1^m)$.

The remainder of this paper is organized as follows. In Section~\ref{sec:proof-theor-refthm:1} we give a simple proof of Theorem~\ref{thm:1} using a Jacobi--Trudi type determinant identity. We also prove Proposition~\ref{prop:KS}. In Section~\ref{sec:proof} we prove Theorem~\ref{thm:main} also using a Jacobi--Trudi identity. 
The main tool of Sections~\ref{sec:proof-theor-refthm:1} and \ref{sec:proof} is to transform the tableaux into some system of  non-intersecting paths on a planar graph and use the Lindstr\"om--Gessel--Viennot lemma \cite{GesselViennot}. 
In Section~\ref{sec:bijective-proof-main} we give a bijective proof of Theorem~\ref{thm:main}. 
In Section~\ref{sec:conn-betw-ssct} we find a connection of our bijection with the bijections due to Novelli, Pak, and Stoyanovskii \cite{NPS} and Krattenthaler \cite{Kratt1999}. Finally, in Section~\ref{sec:final} we provide some open problems.

{\bf Acknowledgements.} The authors want to thank the University of California, Berkeley, where this
work was started and BIRS (Banff Canada), where this work was completed. We want to thank personally
the Director of BIRS that accepted to extend our stay at BIRS after the workshop ``Asymptotic
Algebraic Combinatorics" in March 2019.  Both authors would like to thank Curtis Greene and Carla
Savage for their precious comments and advice during the elaboration of this paper, Brendon Rhoades,
Travis Scrimshaw and U-Keun Song for helpful discussions, and the anonymous referees for their
careful reading and helpful comments.  J.S.K. was supported by NRF grants
\#2019R1F1A1059081 and \#2016R1A5A1008055.

\section{Jacobi--Trudi identity}
\label{sec:proof-theor-refthm:1}

In this section we interpret an $n$-lecture hall tableau as non-intersecting lattice paths and give a Jacobi--Trudi type identity for the generating function $L_\lm^n(\vec x)$ for  $n$-lecture hall tableaux of a given shape. We then prove Theorem~\ref{thm:1} and Proposition~\ref{prop:KS}. 

The paths we consider are on an infinite directed graph embedded in the plane $\mathbb R^2$ defined as follows.

\begin{defn}
The \emph{lecture hall graph} $\GG=(V,E)$ is a directed graph on the vertex set
\[
V = \left\{\left(i,\frac{j}{i+1}\right): i,j\in\NN \right\},
\]
whose edge set $E$ consists of
\begin{itemize}
\item \emph{(nearly) horizontal edges} from $(i,k+\frac{r}{i+1})$ to  $(i+1,k+\frac{r}{i+2})$  for $i,k\in\NN$ and $0\le r\le i$, and
\item \emph{vertical edges} from $(i,k+\frac{r+1}{i+1})$ to $(i,k+\frac{r}{i+1})$ for $i,k\in\NN$ and $0\le r\le i$.
\end{itemize}
\end{defn}

See Figure~\ref{fig:LHL} for an example of the lecture hall graph $\GG$.  We note that in \cite{LHT} a slightly different graph is used to describe lecture hall tableaux, however, both graphs can equally be used for this purpose.

\begin{figure}
  \centering
  \begin{tikzpicture}
\LHLL{9}3    
  \end{tikzpicture}
  \caption{The lecture hall graph $\GG$.}
  \label{fig:LHL}
\end{figure}

We now consider (directed) paths in the lecture hall graph. 
 A \emph{path} in $\GG$ is a (possibly infinite) sequence $P$ of vertices of $\GG$ such that
$(u,v)$ is a directed edge of $\GG$ for every two consecutive elements $u$ and $v$ in $P$. 
If $P$ is a finite path $(u_\ell,u_{\ell-1},\dots,u_1)$, we say that $P$ is a path from $u_\ell$ to $u_1$.
If $P$ is an infinite path $(\ldots,u_3,u_2,u_1)$ for $u_i=(a_i,b_i)$, $i\ge1$, such that 
$\lim_{i\to \infty} a_i=a$, we say that $P$ is a path from $(a,\infty)$ to $(a_1,b_1)$. 

From now on every path considered in this section will be either a finite path
or an infinite path in $\GG$ satisfying the above limit condition.

We define the \emph{weight} $\wt(P)$ of a path $P$ to be the product of its edge weights, where the weight of the horizontal edge from $(i,k+\frac{r}{i+1})$ to  $(i+1,k+\frac{r}{i+2})$ is defined to be $x_k$ and the weight of every vertical edge is defined to be $1$. 
A sequence $(P_1,\dots,P_k)$ of paths is said to be 
 \emph{non-intersecting} if they do not share any vertex.
The weight of the system $(P_1,\dots,P_k)$ of paths 
is defined to be the product $\prod_{i=1}^k {\rm wt}(P_i)$ of the weights of the paths. 
The following lemma gives a way of understanding lecture hall tableaux as non-intersecting paths. 

\begin{lem}\label{lem:LP}
Let $\lambda$ and $\mu$ be partitions satisfying $\mu\subset\lambda$ and $\ell=\ell(\lambda)\le n$. 
Then there exists a bijection between $\LHT_{n}(\lm)$ and the set of non-intersecting paths $(P_1,\ldots ,P_\ell)$
where $P_i$ is a path from $(\mu_i+n-i,\infty)$ to $(\lambda_i+n-i,0)$. 
This bijection is such that if $L\in \LHT_{n}(\lm)$ corresponds to $(P_1,\ldots ,P_\ell)$ then
$$
\vec x^{\lfloor L\rfloor}=\prod_{i=1}^\ell {\rm wt}(P_i).
$$
\end{lem}
\begin{proof}
As in \cite{LHT} the bijection between lecture hall tableaux $L$ and non-intersecting paths $(P_1,\ldots ,P_\ell)$ is constructed
by counting the number of regions under each horizontal edge of each path. Namely, 
$L(i,j)$ is given by the number of regions under the $(j-\mu_i)^{th}$ horizontal edge of $P_i$.
Then the weight of the edge is $x_{\lfloor L(i,j)/(n-i+j)\rfloor}$, so the bijection satisfies the desired property.
\end{proof}

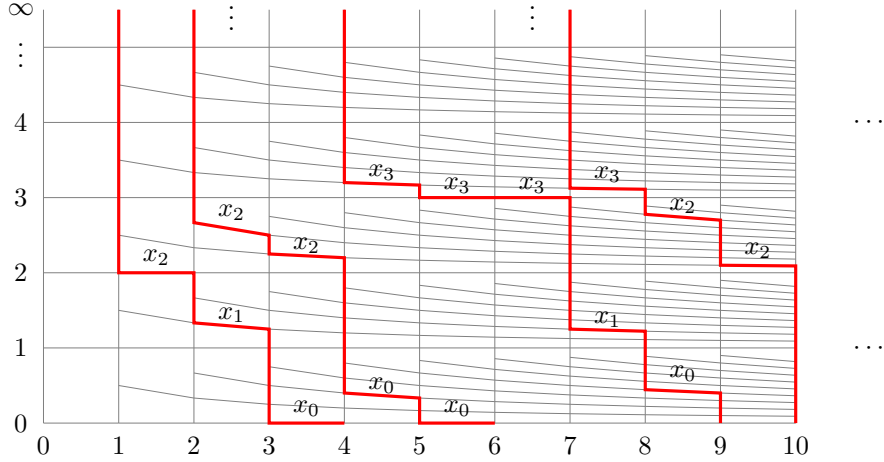
\begin{figure}
  \centering
\begin{tikzpicture}
\LHLL{10}5
\draw [red,very thick](1,5.5)--(1,2)--(2,2)--(2,4/3)--(3,5/4)--(3,0)--(4,0);
\node at (1.5,11/5) {$x_2$};
\node at (2.5,7.2/5) {$x_1$};
\node at (3.5,1/5) {$x_0$};
\draw [red,very thick](2,5.5)--(2,8/3)--(3,10/4)--(3,9/4)--(4,11/5)--(4,2/5)--(5,2/6)--(5,0)--(6,0);
\node at (2.5,16.8/6) {$x_2$};
\node at (3.5,14.2/6) {$x_2$};
\node at (4.5,4/8) {$x_0$};
\node at (5.5,1/5) {$x_0$};
\draw [red,very thick](4,5.5)--(4,16/5)--(5,19/6)--(5,3)--(6,3)--(7,3)--(7,10/8)--(8,11/9)--(8,4/9)--(9,4/10)--(9,0);
\node at (4.5,20/6) {$x_3$};
\node at (5.5,19/6) {$x_3$};
\node at (6.5,19/6) {$x_3$};
\node at (7.5,11/8) {$x_1$};
\node at (8.5,2/3) {$x_0$};
\draw [red,very thick](7,5.5)--(7,25/8)--(8,28/9)--(8,25/9)--(9,27/10)--(9,21/10)--(10,23/11)--(10,0);
\node at (7.5,26.5/8) {$x_3$};
\node at (8.5,26/9) {$x_2$};
\node at (9.5,23/10) {$x_2$};
\end{tikzpicture}
\caption{Non-intersecting paths in $\GG$. For each horizontal edge, its weight is shown above it.}
\label{fig:NLP}
\end{figure}

Figure~\ref{fig:NLP} shows the non-intersecting paths in $\GG$ corresponding to the lecture hall tableau $L\in\LHT_{n}(\lm)$ in Figure~\ref{fig:LHT1}
for $n=5$, $\lambda=(6,6,4,3)$ and $\mu=(3,1)$.
The paths in Figure~\ref{fig:NLP} have weight $x_0^4x_1^2x_2^5x_3^4$, which is equal to ${\bf x}^{\lfloor L\rfloor}$. 
The entries of $\lfloor L\rfloor$ can be seen on the right of Figure~\ref{fig:LHT1}.

Recall that $\vec x=(x_0,x_1,\dots )$ and that $|\vec x|=x_0+x_1+\cdots $. 
The following proposition is a Jacobi--Trudi type identity for $L_\lm^n(\vec x)$.
\begin{prop}\label{prop:JT}
Let $\lambda$ and $\mu$ be partitions satisfying $\mu\subset\lambda$ and $\ell=\ell(\lambda)\le n$. 
Then we have
\[
L_\lm^n(\vec x)  = \det
\left(L^{\mu_j+n-j+1}_{(\lambda_i-\mu_j-i+j)}(\vec x) \right)_{1\le i,j\le\ell}.
\]
\label{NILP}
\end{prop}
\begin{proof}
This is a direct consequence of the Lindstr\"om--Gessel--Viennot lemma \cite{GesselViennot}, which states that the weight generating function for non-intersecting paths
from vertices $u_1,u_2,\ldots ,u_\ell$ to vertices $v_1,v_2,\ldots ,v_\ell$ of the planar graph $\GG$
is
$$
\det (P(u_j,v_i))_{1\le i,j\le \ell},$$
where 
$P(u_j,v_i)$ is the weight generating function of the paths from $u_j$ to $v_i$.
Here we choose
$u_j=(\mu_j+n-j,\infty)$, $v_i=(\lambda_i+n-i,0)$
and  therefore $P(u_j,v_i)=L^{\mu_j+n-j+1}_{(\lambda_i-\mu_j-i+j)}(\vec x)$.
Then the proposition follows from Lemma~\ref{lem:LP}. 
\end{proof}

Let us now compute the entries of the matrix of the previous proposition:

\begin{prop}\label{prop:one}
For $n,k\ge0$ we have
$$
L^{n}_{(k)}(\vec x)=|\vec x|^{k}{n+k-1\choose k}.   
$$
\end{prop}
\begin{proof}
Let us first recall that
\begin{equation}
  \label{eq:8}
L^{n}_{(k)}(\vec x)=\sum_L \vec x^{\lfloor L\rfloor},  
\end{equation}
where the sum is over the $n$-lecture hall tableaux $L\in\LHT_n(\lambda)$ of shape $\lambda=(k)$, i.e.,
$$
\frac{L(1,1)}{n}\ge \frac{L(1,2)}{n+1}\ge \dots \ge \frac{L(1,k)}{n+k-1}\ge 0.
$$

Consider the case $\vec x=(x_0,0,0,\ldots )$. Then the $n$-lecture hall tableaux $L$ contributing nonzero terms in \eqref{eq:8} are those satisfying
$$
1>\frac{L(1,1)}{n}\ge \frac{L(1,2)}{n+1}\ge \dots \ge \frac{L(1,k)}{n+k-1}\ge 0.
$$
It is easy to check that for $a,b,k\in\NN$, the condition $1>\frac ak\ge \frac b{k+1}$
is equivalent to $k>a\ge b$. Thus, the above condition is equivalent to
$n > L(1,1)\ge \dots \ge{L(1,k)}\ge 0$ and we have
$$
L^{n}_{(k)}(x_0):=L^{n}_{(k)}(x_0,0,0,\dots)=x_0^k {n+k-1\choose k}.
$$

Now consider the general case $\vec x=(x_0,x_1,\dots)$.  
Fix an  $n$-lecture hall tableau $L\in\LHT_n((k))$ and 
let $j$ be the index such that $\frac{L(1,j)}{n+j-1}\ge 1$ and $\frac{L(1,j+1)}{n+j}<1$.
Here we suppose that $L_{1,0}=\infty$ and $L_{1,k+1}=0$ so that the index $0\le j\le k$ is always defined.
We can decompose $L$ into two lecture hall tableaux $L'\in\LHT_n((j))$
and $L''\in\LHT_{n+j}((k-j))$ so that $L'(1,i)=L(1,i)$ and $L''(1,i)=L(1,j+i)$.
Then $L'$ and $L''$ satisfy
\begin{equation}
  \label{eq:L'}
\frac{L'(1,1)}{n}\ge \dots \ge\frac{L'(1,j)}{n+j-1}\ge 1,  
\end{equation}
\begin{equation}
  \label{eq:L''}
1> \frac{L''(1,1)}{n+j}\ge \dots \ge\frac{L''(1,k-j)}{n+k-1}\ge 0.
\end{equation}
Conversely, for any pair of $L'$ and $L''$ satisfying \eqref{eq:L'} and \eqref{eq:L''}, we obtain an $n$-lecture hall tableau $L\in\LHT_n((k))$. 
Moreover, the tableaux $L'\in\LHT_n((j))$ satisfying the condition~\eqref{eq:L'} are those
contributing nonzero terms in $L^{n}_{(j)}(0,x_1,x_{2},\ldots )$ and
the tableaux $L''\in\LHT_{n+j}((k-j))$ satisfying the condition~\eqref{eq:L''} are those
contributing nonzero terms in $L^{n+j}_{(k-j)}(x_0)$. Therefore
$$
L^{n}_{(k)}(x_0,x_1,\dots)=\sum_{j=0}^k L^{n+j}_{(k-j)}(x_0) L^{n}_{(j)}(0,x_1,x_2,\ldots).
$$

Now we notice that sequences $(L'(1,1),\ldots , L'(1,j))$ such that
\begin{equation*}
\frac{L'(1,1)}{n}\ge \dots \ge\frac{L'(1,j)}{n+j-1}\ge 1,  
\end{equation*}
are in bijection with sequences  $(U'(1,1),\ldots , U'(1,j))$ such that
\begin{equation*}
\frac{U'(1,1)}{n}\ge \dots \ge\frac{U'(1,j)}{n+j-1}\ge 0,  
\end{equation*}
by setting $U'(1,i)=L'(1,i)-n+i-1$ for $1\le i\le j$.
This implies that 
$$
\vec x^{\lfloor L'\rfloor}=\prod_{i=1}^j x_{\lfloor L'_{1,i}/(n-i+1)\rfloor}=
\prod_{i=1}^j x_{\lfloor U'_{1,i}/(n-i+1)\rfloor+1}.
$$
We get that
$$
 L^{n}_{(j)}(0,x_1,x_2,\ldots)= L^{n}_{(j)}(x_1,x_2,\ldots).
$$

Therefore
$$
L^{n}_{(k)}(x_0,x_1,\dots)=\sum_{j=0}^k L^{n+j}_{(k-j)}(x_0) L^{n}_{(j)}(x_1,x_2,\ldots).
$$
This gives $L^{n}_{(k)}(x_0,x_1,\dots)=\binom{n+k-1}{k}|\vec x|^k$ using induction.
\end{proof}

Combining the two previous propositions, we obtain the main theorem in this section.

\begin{thm}\label{thm:L}
Let $\lambda$ and $\mu$ be partitions satisfying $\mu\subset\lambda$ and $\ell=\ell(\lambda)\le n$. 
Then we have
\[
L_\lm^{n}(\vec x) = |\vec x|^{|\lm|} \det\left(\binom{\lambda_i+n-i}{\mu_j+n-j}\right)_{1\le i,j\le \ell} .
\]
\end{thm}
\begin{proof}
By Propositions~\ref{prop:JT} and \ref{prop:one}, we have
\[
L_\lm^n(\vec x)  = \det
\left(|\vec x|^{\lambda_i-i-\mu_j+j}\binom{\lambda_i+n-i}{\mu_j+n-j} \right) _{1\le i,j\le \ell}.
\]
By factoring out the factor $|\vec x|^{\lambda_i-i}$ for each row $i$
and the factor $|\vec x|^{j-\mu_j}$ for each column $j$, we obtain the theorem.
\end{proof}

Setting $\vec x=(1^m)$ in Theorem~\ref{thm:L}, we obtain Theorem \ref{thm:1}.
Since $|\SSCT_n(\lm)| =L^n_\lm(1,0,0,\dots)$, Theorem~\ref{thm:L} implies that
\begin{equation}
  \label{eq:14}
|\SSCT_n(\lm)|=\det\left(\binom{\lambda_i+n-i}{\mu_j+n-j}\right)_{1\le i,j\le \ell}.  
\end{equation}
Therefore Theorem~\ref{thm:L} is equivalent to
\begin{equation}
  \label{eq:15}
L_\lm^{n}(\vec x) = |\vec x|^{|\lm|}  |\SSCT_n(\lm)|.
\end{equation}
Since $\SSCT_n(\lambda) =\SSYT_n(\lambda)$,
by setting $\mu=\emptyset$ in \eqref{eq:15} we obtain the following corollary.

\begin{cor}\label{thm:main2}
For a partition $\lambda$ with at most $n$ parts, we have
\[
L_\lambda^n(\vec x) =  |\vec x|^{|\lambda|} s_\lambda(1^n).
\]
\end{cor}

We finish this section by giving a proof of Proposition~\ref{prop:KS}.

\begin{proof}[Proof of Proposition~\ref{prop:KS} ]
The first equality is shown in \eqref{eq:14}. It remains to show that
\begin{equation}
  \label{eq:1}
\det\left(\binom{\lambda_i+n-i}{\mu_j+n-j}\right)_{1\le i,j\le \ell} = \frac{|\SYT(\lm)|}{|\lm|!} \prod_{x\in\lm} (n+c(x)).
\end{equation}
We need the following determinant formula for $|\SYT(\lm)|$ due to Aitken \cite{Aitken}, see also \cite[Corollary~7.16.3]{Stanley}:
\[
|\SYT(\lm)|=|\lm|! \det\left(\frac{1}{(\lambda_i-\mu_j-i+j)!}\right) _{1\le i,j\le \ell}.
\]
Then \eqref{eq:1} follows immediately from Aitken's formula with the identities:
\[
\det\left(\binom{\lambda_i+n-i}{\mu_j+n-j}\right)_{1\le i,j\le \ell}
=\prod_{i=1}^{\ell(\lambda)} \frac{(\lambda_i+n-i)!}{(\mu_i+n-i)!}
\det\left(\frac{1}{(\lambda_i-\mu_j-i+j)!}\right) _{1\le i,j\le \ell},
\]
and
\[
\prod_{i=1}^{\ell(\lambda)} \frac{(\lambda_i+n-i)!}{(\mu_i+n-i)!}
=\prod_{x\in\lm}(n+c(x)),
\]
which can be easily verified.
\end{proof}

\section{Proof of Theorem~\ref{thm:main} using Jacobi--Trudi identity}
\label{sec:proof}

In this section, we prove Theorem~\ref{thm:main} using a Jacobi--Trudi identity for the generating function
\begin{equation}
  \label{eq:12}
S_{\lm}^{n}(\vec y) = \sum_{T\in\SSCT_{n}(\lm)} \vec y^{T}.  
\end{equation}
To this end we introduce another infinite directed graph. We use the notation $\omega$ for the smallest infinite ordinal number, i.e, $1<2<\cdots<\omega$. 

\begin{defn}
The \emph{content graph} $\GG'$ is the directed graph $\GG'=(V',E')$ 
on the vertex set
\[
V' = \left\{\left(i,\omega+\frac{r}{i+1}\right): i\in\NN, r\in\{0,1,\dots,i+1\}\right\},
\]
whose edge set $E'$ consists of
\begin{itemize}
\item \emph{(nearly) horizontal edges} from $(i,\omega+\frac{r}{i+1})$ to  $(i+1,\omega+\frac{r}{i+2})$  for $i\in\NN$ and $0\le r\le i$, and
\item \emph{vertical edges} from $(i,\omega+\frac{r+1}{i+1})$ to $(i,\omega+\frac{r}{i+1})$ for $i\in\NN$ and $0\le r\le i$.
\end{itemize}
\end{defn}

\begin{figure}
  \centering
\begin{tikzpicture}
\wLHLL{10}
\end{tikzpicture}
\caption{The content graph $\GG'$.}
\label{fig:GG'}
\end{figure}
 
Figure~\ref{fig:GG'} shows the content graph $\GG'$. Now to any path $P'$ in $\GG'$, we associate a monomial $\wt(P')$ equal to the product of the weights of the edges of $P'$, where the weight of the horizontal edge from $(i,\omega+\frac{r}{i+1})$ to $(i+1,\omega+\frac{r}{i+2})$ is defined to be $y_r$ and the weight of every vertical edge is $1$.

The following lemma gives a way to understand a semistandard $n$-content tableau as non-intersecting paths in $\GG'$.

\begin{lem}\label{lem:S}
There is a bijection between   
$\SSCT_n(\lm)$ and the set of non-intersecting paths $(P_1,\ldots ,P_\ell)$ in $\GG'$,
where each $P_i$ starts at $u_i=(\mu_i+n-i,\omega+1)$ and ends at $v_i=(\lambda_i+n-i,\omega)$.
The correspondence between $T\in\SSCT_n(\lm)$ and $(P_1,\ldots ,P_\ell)$ is as follows. 
The number of regions under the $(j-\mu_i)^{th}$ horizontal step
of $P_i$ is the entry $T(i,j)$. In this case we have $\vec y^T=\prod_{i=1}^n {\rm wt}(P_i)$.
\end{lem}
\begin{proof}
This can be proved similarly to the proof of Lemma~\ref{lem:LP}.
\end{proof}

For example, the non-intersecting paths on Figure~\ref{fig:tilde1} correspond to the tableau on Figure~\ref{fig:ssct}. Note that both have weight $y_0^4y_1^4y_2y_3^2y_4$.

\begin{figure}
  \centering
\begin{tikzpicture}
\wLHLL{10}
\draw [red,very thick] (0,0)--(0,0);
\draw [red,very thick] (1,1)--(1,1/2)--(2,1/3)--(2,0)--(4,0);
\draw [red,very thick](2,1)--(2,2/3)--(3,2/4)--(3,1/4)--(4,1/5)--(4,1/5)--(5,1/6)--(6,1/7)--(6,0);
\draw [red,very thick](4,1)--(4,4/5)--(5,4/6)--(5,3/6)--(6,3/7)--(6,3/7)--(7,3/8)--(7,0)--(9,0);
\node at (1.5,5/8) {$y_1$};
\node at (2.5,1/8) {$y_0$};
\node at (3.5,1/8) {$y_0$};
\node at (2.5,7/9) {$y_2$};
\node at (3.5,3/7) {$y_1$};
\node at (4.5,3/8) {$y_1$};
\node at (5.5,2.5/8) {$y_1$};
\node at (4.5,5.5/6) {$y_4$};
\node at (5.5,5.6/9) {$y_3$};
\node at (6.5,4.5/8) {$y_3$};
\node at (7.5,1.5/8) {$y_0$};
\node at (8.5,1.5/8) {$y_0$};
\end{tikzpicture}
\caption{Non-intersecting paths in $\GG'$. The weight of each horizontal edge is shown above the edge.}
\label{fig:tilde1}
\end{figure}
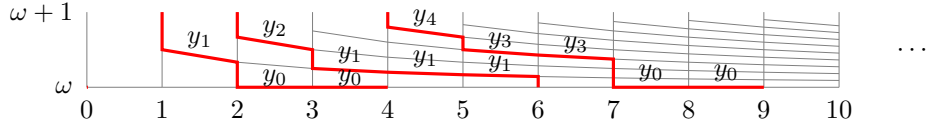

\begin{figure}
  \centering
\begin{tikzpicture}[scale=.6]
\cell22{4} \cell23{3} \cell24{3} \cell25{0} \cell260
\cell312 \cell321 \cell331 \cell341
\cell411 \cell420 \cell430
\draw (0,-4)--(0,-2)--(1,-2)--(1,-1)--(6,-1);
\end{tikzpicture} 
  \caption{A semistandard $n$-content tableau of shape $\lm$ with  $n=4$, $\lambda=(6,4,3)$ and $\mu=(1)$.}
  \label{fig:ssct}
\end{figure}

The following is a Jacobi--Trudi identity for $S_{\lm}^{n}(\vec x)$. 

\begin{prop}\label{prop:SJT}
Let $\lambda$ and $\mu$ be partitions satisfying $\mu\subset\lambda$ and $\ell=\ell(\lambda)\le n$. 
Then we have
\[
S_\lm^n(\vec y)=\det\left(S^{\mu_j+n-j+1}_{(\lambda_i-\mu_j-i+j)}(\vec y) \right)_{1\le i,j\le\ell}.
\]
\end{prop}
\begin{proof}
The proof is similar to that of Proposition~\ref{prop:JT}, hence we omit it.
\end{proof}

Since $\SSCT_n(\lambda)=\SSYT_n(\lambda)$, the definition \eqref{eq:12} of $S_{\lm}^{n}(\vec x)$ implies that
for $k\ge0$ and $n\ge1$,
\begin{equation}
  \label{eq:sh}
S^n_{(k)}(\vec y)=s_{(k)}(y_0,\dots,y_{n-1}) =h_k(y_0,\ldots ,y_{n-1}),
\end{equation}
where $h_k(y_0,\dots,y_{n-1})$ is the \emph{complete homogeneous polynomial} defined by
\begin{equation}
  \label{eq:hk}
h_k(y_0,\dots,y_{n-1}) = \sum_{0\le i_1\le\dots\le i_k\le n-1} y_{i_1}\cdots y_{i_k}.  
\end{equation}
Note that $y_0,\dots,y_{n-1}$ are the only variables that actually appear in $S^n_{(k)}(\vec y)$ even though $\vec y=(y_0,y_1,\dots)$ is an infinite sequence of variables. 
Using \eqref{eq:sh}, Proposition~\ref{prop:SJT} can be restated as
\begin{equation}
  \label{eq:SJT2}
S_\lm^n(\vec y) =\det(h_{\lambda_i-\mu_j-i+j}(y_0,\ldots, y_{\mu_j+n-j})) _{1\le i,j\le\ell}.
\end{equation}

In order to prove Theorem~\ref{thm:main} we introduce yet another graph. 

\begin{defn}
The \emph{extended lecture hall graph} $\GG^*$ is the disjoint union $\GG'\uplus \GG$ of the content graph $\GG'$ and the lecture hall graph $\GG$.
\end{defn}

\begin{figure}
  \centering
  \begin{tikzpicture}
\ELHLL{9}3    
  \end{tikzpicture}
  \caption{The extended lecture hall graph $\GG^*=\GG'\uplus\GG$.}
  \label{fig:ELHL}
\end{figure}
 
We will draw the extended lecture hall graph $\GG^*=\GG'\uplus \GG$ with $\GG'$ on top of $\GG$ as shown in Figure~\ref{fig:ELHL} so that each vertex $(i,\omega)$ of $\GG'$ can be considered as the ``limit'' of the sequence of vertices $(i,0),(i,1),(i,2),\dots$ in $\GG$.

We define an \emph{$\omega$-path} to be a pair $Q=(P',P)$ satisfying the following conditions:
\begin{itemize}
\item $P'$ and $P$ are paths in $\GG'$ and $\GG$, respectively.
\item $P'$ is a path from $(a,\omega+1)$ to $(b,\omega)$ and
$P$ is a path from $(b,\infty)$ to $(c,0)$ for some $a\le b\le c$. 
\end{itemize}
In this case we say that $Q$ is an $\omega$-path from $(a,\omega+1)$ to $(c,0)$. 
We use the weight $\wt(P)$ for a path in $\GG$ as in Section~\ref{sec:proof-theor-refthm:1}
and define the weight of an $\omega$-path $Q=(P',P)$ by $\wt(Q)=\wt(P')\wt(P)$.

We are now ready to prove Theorem \ref{thm:main}, which states that
\begin{equation}
  \label{eq:11}
S^n_{\lm}(|\vec x|+\vec y) = \sum_{\mu\subset\nu\subset\lambda} L_{\lambda/\nu}^{n}(\vec x) S^n_{\nu/\mu}(\vec y).  
\end{equation}

\begin{proof}[Proof of Theorem \ref{thm:main}]
Let $LHS$ and $RHS$ be the left hand side and the right hand side of \eqref{eq:11}, respectively. By \eqref{eq:SJT2}, we have
\begin{equation}
  \label{eq:13}
LHS =\det(h_{\lambda_i-\mu_j-i+j}(y_0+|\vec x|,\ldots ,y_{\mu_j+n-j}+|\vec x|))_{1\le i,j\le\ell}.
\end{equation}
Our strategy is to express $RHS$ also as a determinant
that agrees with the determinant in \eqref{eq:13} entry-wise.

First, observe that
\[
RHS=\sum_{\mu\subset\nu\subset\lambda} L_{\lambda/\nu}^{n}(\vec x) S^n_{\nu/\mu}(\vec y) 
=\sum_{(L,S)} \vec x^{\flr L} \vec y^S,
\]
where the sum is over all pairs $(L,S)$ of tableaux $L\in\LHT_n(\lambda/\nu)$
and $S\in\SSCT_n(\nu/\mu)$ for some partition $\nu$ with $\mu\subset\nu\subset\lambda$. 
Combining the bijections in Lemmas~\ref{lem:LP} and \ref{lem:S}, 
we obtain a bijection between the set of such pairs $(L,S)$
and the set of non-intersecting $\omega$-paths $(Q_1,\dots,Q_\ell)$ such that
$Q_i$ is an $\omega$-path from $u_i=(\mu_i+n-i,\omega+1)$ to $v_i=(\lambda_i+n-i,0)$. Furthermore, under this bijection we have
$\vec x^{\flr L} \vec y^S=\wt(Q_1)\cdots \wt(Q_\ell)$, which implies that
\[
RHS=\sum_{(L,S)} \vec x^{\flr L} \vec y^S =\sum_{(Q_1,\dots,Q_\ell)} \wt(Q_1)\cdots \wt(Q_\ell).  
\]
For example, the pair $(L,S)$ of tableaux given on Figure~\ref{fig:LS} corresponds to the non-intersecting $\omega$-paths
on Figure~\ref{fig:tilde}.

\begin{figure}
  \centering
\begin{tikzpicture}[scale=.6]
\cell14{$1$} \cell15{$25$} \cell16{$21$}
\cell22{$1$} \cell23{$0$} \cell24{$21$} \cell25{$10$} \cell26{$4$}
\cell31{$8$} \cell32{$9$} \cell33{$2$} \cell34{$0$}
\cell41{$4$} \cell42{$4$} \cell43{$0$}
\draw (0,-4)--(0,-2)--(1,-2)--(1,-1)--(3,-1)--(3,0)--(6,0);
\draw [very thick] (0,-2)--(3,-2)--(3,-1)--(4,-1)--(4,0);
\end{tikzpicture}
  \caption{A pair $(L,S)$ of tableaux $L\in\LHT_n(\lambda/\nu)$ and $S\in\SSCT_n(\nu/\mu)$ for $n=5$, $\lambda=(6,6,4,3)$, $\mu=(3,1)$, and $\nu=(4,3)$. The tableaux $L$ and $S$ are separated by the thick border.}
  \label{fig:LS}
\end{figure}

\begin{figure}
  \centering
\begin{tikzpicture}
\ELHLL{10}5
\draw [red,very thick](1,6.5)--(1,2)--(2,2)--(2,4/3)--(3,5/4)--(3,0)--(4,0);
\node at (1.5,11/5) {$x_2$};
\node at (2.5,7.2/5) {$x_1$};
\node at (3.5,1/5) {$x_0$};
\draw [red,very thick](2,6.5)--(2,8/3)--(3,10/4)--(3,9/4)--(4,11/5)--(4,2/5)--(5,2/6)--(5,0)--(6,0);
\node at (2.5,17/6) {$x_2$};
\node at (3.5,.1+14/6) {$x_2$};
\node at (4.5,4.1/8) {$x_0$};
\node at (5.5,1/5) {$x_0$};
\draw [red,very thick](4,6.5)--(4,1.5+21/5)--(5,1.5+25/6)--(5,1.5+4)--(6,1.5+4)--(6,3)--(7,3)--(7,10/8)--(8,11/9)--(8,4/9)--(9,4/10)--(9,0);
\node at (4.5,1.6+26/6) {$y_1$};
\node at (5.5,1.5+25/6) {$y_0$};
\node at (6.5,25.2/8) {$x_3$};
\node at (7.5,11/8) {$x_1$};
\node at (8.5,.1+1/2) {$x_0$};
\draw [red,very thick](7,6.5)--(7,1.5+33/8)--(8,1.5+37/9)--(8,25/9)--(9,27/10)--(9,21/10)--(10,23/11)--(10,0);
\node at (7.5,1.6+34/8) {$y_1$};
\node at (8.5,26/9) {$x_2$};
\node at (9.5,.1+22/10) {$x_2$};
\end{tikzpicture}
\caption{Non-intersecting $\omega$-paths in $\GG^*$. For each horizontal edge, its weight is shown above it.}
\label{fig:tilde}
\end{figure}

By the Lindstr\"om--Gessel--Viennot lemma, we have
\begin{equation}
  \label{eq:19}
RHS=\sum_{(Q_1,\dots,Q_\ell)} \wt(Q_1)\cdots \wt(Q_\ell)=\det(\tilde{P}(u_j,v_i))_{1\le i,j\le \ell},  
\end{equation}
where $\tilde{P}(u_j,v_i)$ is the sum of $\wt(Q)$ for all $\omega$-paths $Q$
from $u_j=(\mu_j+n-j,\omega+1)$ to $v_i=(\lambda_i+n-i,0)$. It is easy to see that
\begin{equation}
  \label{eq:18}
\tilde{P}(u_j,v_i)=\sum_{k=0}^{\lambda_i-\mu_j-i+j}
S^{\mu_j+n-j+1}_{(k)}(\vec y) L^{\mu_j+n-j+k+1}_{(\lambda_i-\mu_j-i+j-k)}(\vec x).
\end{equation}
By \eqref{eq:sh} and Proposition~\ref{prop:one}, we have
\begin{align*}
S^{\mu_j+n-j+1}_{(k)}(\vec y) &=h_{k}(y_0,\ldots ,y_{\mu_j+n-j}),\\
 L^{\mu_j+n-j+k+1}_{(\lambda_i-\mu_j-i+j-k)}(\vec x)&
=|\vec x|^{\lambda_i-\mu_j-i+j-k}\binom{\lambda_i+n-i}{\lambda_i-\mu_j-i+j-k}.
\end{align*}

Therefore, by \eqref{eq:13}, \eqref{eq:19}, \eqref{eq:18} and the above two equations,
it suffices to prove the following identity:
\begin{multline}
  \label{eq:17}
h_{\lambda_i-\mu_j-i+j}(y_0+|\vec x|,\ldots ,y_{\mu_j+n-j}+|\vec x|) \\
=\sum_{k=0}^{\lambda_i-\mu_j-i+j} h_{k}(y_0,\ldots ,y_{\mu_j+n-j})|\vec x|^{\lambda_i-\mu_j-i+j-k}\binom{\lambda_i+n-i}{\lambda_i-\mu_j-i+j-k}.
\end{multline}

Using the definition \eqref{eq:hk} of the complete homogeneous polynomial, it is not hard to see that
\begin{eqnarray*}
h_{t}(y_0+|\vec x|,\ldots ,y_{a}+|\vec x|)&=&\sum_{0\le i_1\le \dots \le i_t\le a} (y_{i_1}+|\vec x|)(y_{i_2}+|\vec x|)\dots (y_{i_t}+|\vec x|)\\
&=&\sum_{k=0}^t h_{k}(y_0,\ldots ,y_{a})h_{t-k}(|\vec x|^{a+k+1}),
\end{eqnarray*}
where $h_{t-k}(|\vec x|^{a+k+1})$ means $h_{t-k}(\overbrace{|\vec x|,\dots,|\vec x|}^{a+k+1})$.
Since $h_{t-k}(|\vec x|^{a+k+1}) =|\vec x|^{t-k}\binom{a+t}{t-k}$, we obtain \eqref{eq:17}
from the above identity by setting $a=\mu_j+n-j$ and $t=\lambda_i-\mu_j-i+j$.
The proof is now complete.
\end{proof}

\section{A bijective proof of the main theorem}
\label{sec:bijective-proof-main}

In this section we give a bijective proof of Theorem~\ref{thm:main}. We first introduce some definitions and restate the theorem accordingly.

\begin{figure}
  \centering
\begin{tikzpicture}[scale=.6]
\cell14{25} \cell15{25} \cell16{21}
\cell22{16} \cell23{18} \cell24{21} \cell25{10} \cell264
\cell318 \cell329 \cell332 \cell340
\cell414 \cell424 \cell430
\draw (0,-4)--(0,-2)--(1,-2)--(1,-1)--(3,-1)--(3,0)--(6,0);
\end{tikzpicture} \qquad \qquad
\begin{tikzpicture}[scale=.6]
\cell14{$\frac{25}8$} \cell15{$\frac{25}9$} \cell16{$\frac{21}{10}$}
\cell22{$\frac{16}5$} \cell23{$\frac{18}6$} \cell24{$\frac{21}7$} \cell25{$\frac{10}8$} \cell26{$\frac49$}
\cell31{$\frac83$} \cell32{$\frac94$} \cell33{$\frac25$} \cell34{$\frac06$}
\cell41{$\frac42$} \cell42{$\frac43$} \cell43{$\frac04$}
\draw (0,-4)--(0,-2)--(1,-2)--(1,-1)--(3,-1)--(3,0)--(6,0);
\end{tikzpicture}\qquad \qquad
\begin{tikzpicture}[scale=.6]
\cell14{$1_3$} \cell15{$7_2$} \cell16{$1_2$}
\cell22{$1_3$} \cell23{$0_3$} \cell24{$0_3$} \cell25{$2_1$} \cell26{$4_0$}
\cell31{$2_2$} \cell32{$1_2$} \cell33{$2_0$} \cell34{$0_0$}
\cell41{$0_2$} \cell42{$1_1$} \cell43{$0_0$}
\draw (0,-4)--(0,-2)--(1,-2)--(1,-1)--(3,-1)--(3,0)--(6,0);
\end{tikzpicture}
  \caption{On the left is a lecture hall tableau $L\in \LHT_n(\lm)$ for $n=5$, $\lambda=(6,6,4,3)$ and $\mu=(3,1)$. 
The diagram in the middle shows the number $L(i,j)/(n+c(i,j))$ for each entry $(i,j)\in\lm$. The diagram on the right is the corresponding marked tableau $T$, given by $T(i,j)=a_r$, where $a$ and $r$ are the unique integers satisfying $L(i,j)=r\cdot(n-i+j)+a$ and $0\le a<n-i+j$.}
  \label{fig:LHT}
\end{figure}

A \emph{marked tableau} of shape $\lm$ is a map $T:\lm\to \NN\times (\NN\cup \{\infty\})$. If $T(i,j)=(a,r)$ we say that $a$ is a \emph{value} and $r$ is a \emph{mark}. Instead of $T(i,j)=(a,r)$, we will simply write $T(i,j)=a_r$.  A \emph{marked $n$-content tableau} is a marked tableau $T$ with a condition that if $T(i,j)=a_r$, then $0\le a<n-i+j$. For a marked tableau $T$ of shape $\lm$ and a skew shape $\alpha\subset\lm$, we denote by
$T|_{\alpha}$ the restriction of $T$ to the cells in $\alpha$. 

Let $T$ be a marked tableau of shape $\lm$.
For each $(i,j)\in\lm$, let
\[
\wt^*(T(i,j))=
\begin{cases}
  x_{b}, & \mbox{if $T(i,j)=a_b$ and $b\ne\infty$,}\\
  y_{a}, & \mbox{if $T(i,j)=a_\infty$}.
\end{cases}
\]
The \emph{weight} $\wt^*(T)$ of $T$ is defined by
\[
\wt^*(T)=\prod_{(i,j)\in\lm} \wt^*(T(i,j)).
\]

Consider an $n$-lecture hall tableau $L\in\LHT_n(\lm)$.  We construct a marked tableau $T$ as follows.  For each cell $(i,j)\in\lm$, let $T(i,j)=a_r$, where $r=\flr{L(i,j)/(n+j-i)}$ and $a=L(i,j)-r\cdot(n+j-i)$. See Figure~\ref{fig:LHT}.  Clearly, $L$ can be recovered from $T$. From now on we will identify
the lecture hall tableau $L$ with the marked tableau $T$. Note that under this identification every mark of a lecture hall tableau is a nonnegative integer.

An \emph{extended $n$-lecture hall tableau} of shape $\lm$ is a marked tableau $T:\lm\to \NN\times(\NN\cup\{\infty\})$ satisfying the following conditions:
\begin{enumerate}
\item If $(i,j)\in\lm$ and $T(i,j)=a_r$, then $0\le a<n+j-i$. 
\item If $(i,j),(i,j+1)\in\lm$ and $T(i,j)=a_r$, $T(i,j+1)=b_s$, then we have either $r>s$, or
$r=s$ and $a\ge b$.
\item If $(i,j),(i+1,j)\in\lm$ and $T(i,j)=a_r$, $T(i+1,j)=b_s$, then we have either $r>s$, or
$r=s$ and $a>b$.
\end{enumerate}
We denote by $\LHT^*_n(\lm)$ the set of extended $n$-lecture hall tableaux of shape $\lm$. See Figure~\ref{fig:LHT*} for an example. 

\begin{figure}
  \centering
\begin{tikzpicture}[scale=.6]
\cell14{$1_\infty$} \cell15{$7_2$} \cell16{$1_2$}
\cell22{$1_\infty$} \cell23{$0_\infty$} \cell24{$0_3$} \cell25{$2_1$} \cell26{$4_0$}
\cell31{$2_2$} \cell32{$1_2$} \cell33{$2_0$} \cell34{$0_0$}
\cell41{$0_2$} \cell42{$1_1$} \cell43{$0_0$}
\draw (0,-4)--(0,-2)--(1,-2)--(1,-1)--(3,-1)--(3,0)--(6,0);
\draw[color=blue, line width=1] (3.5,-2.5) circle (.5);
\end{tikzpicture}
  \caption{An extended $n$-lecture hall tableau $L$ in $\LHT^*_n(\lm)$, where $n=5$, $\lambda=(6,6,4,2)$ and $\mu=(3,1)$. The weight of $L$ is
$\wt^*(L)=x_0^4x_1^2x_2^5x_3^1 y_0y_1^2$.
The tail of $L$ is indicated by the blue circle.}
  \label{fig:LHT*}
\end{figure}

A \emph{marked semistandard $n$-content tableau} 
is a marked tableau $T$ such that the tableau obtained from $T$ by deleting its marks
is a semistandard $n$-content tableau. See Figure~\ref{fig:SSCT*} for an example. We denote by $\SSCT^*_n(\lm)$ the set of marked semistandard $n$-content tableaux of shape $\lm$. 
From the definition one can easily see that
\begin{equation}
  \label{eq:2}
S_{\lm}^n(|\vec x|+\vec y) = \sum_{T\in\SSCT^*_n(\lm)} \wt^*(T).  
\end{equation}

\begin{figure}
  \centering
\begin{tikzpicture}[scale=.6]
\cell14{$6_2$} \cell15{$2_1$} \cell16{$1_\infty$}
\cell22{$2_0$} \cell23{$2_2$} \cell24{$2_0$} \cell25{$0_2$} \cell26{$0_3$}
\cell31{$2_2$} \cell32{$1_\infty$} \cell33{$1_0$} \cell34{$0_0$}
\cell41{$0_1$} \cell42{$0_2$} \cell43{$0_\infty$}
\draw (0,-4)--(0,-2)--(1,-2)--(1,-1)--(3,-1)--(3,0)--(6,0);
\draw[color=red, line width=1] (1.5,-2.5) circle (.5);
\end{tikzpicture}
  \caption{A marked semistandard $n$-content tableau $S$ in $\SSCT^*_n(\lm)$, where $n=5$, $\lambda=(6,6,4,2)$ and $\mu=(3,1)$. The weight of $S$ is
$\wt^*(S)=x_0^4x_1^2x_2^5x_3^1 y_0y_1^2$. The head of $S$ is indicated by the red cell.}
  \label{fig:SSCT*}
\end{figure}

Observe that if $T$ is an extended $n$-lecture hall tableau, then the marks are weakly decreasing in each row and each column, and for all $i\in\NN\cup\{\infty\}$ the values with mark $i$ form a semistandard $n$-content tableau. Therefore, if we restrict $T$ to the cells whose marks are not $\infty$, we obtain an $n$-lecture hall tableau, which implies that
\begin{equation}
  \label{eq:5}
\sum_{T\in\LHT^*_n(\lambda/\mu)} \wt^*(T)   =
\sum_{\nu} \left(\sum_{T\in\LHT_n(\lambda/\nu)} \vec x^{\flr T}
\sum_{T\in\SSCT_n(\nu/\mu)} \vec y^T \right)
= \sum_\nu L^n_{\lambda/\nu}(\vec x) S^n_{\nu/\mu}(\vec y).
\end{equation}
By \eqref{eq:2} and \eqref{eq:5}, Theorem~\ref{thm:main} can be restated as follows.

\begin{thm}\label{thm:main2}
We have
\[
\sum_{T\in\LHT^*_n(\lambda/\mu)} \wt^*(T)= \sum_{T\in\SSCT^*_n(\lm)} \wt^*(T).
\]
\end{thm}

We will construct a weight-preserving bijection between
$\LHT^*_n(\lambda/\nu)$ and $\SSCT^*_n(\lm)$. The basic idea is to sort the values of $L\in\LHT^*_n(\lm)$ using a variation of ``jeu de taquin'' according to a certain order of the cells in $\lm$ depending on $L$ itself. Our algorithms are inspired by those due to Krattenthaler \cite{Kratt1999}.

\begin{algorithm}[Value-jeu de taquin]\label{alg:vjdt}
The \emph{value-jdt algorithm} is described as follows. 
  \begin{description}
  \item[Notation] $\vjdt(P,u)=(Q,v)$.
  \item[Input] A pair $(P,u)$ of a marked tableau $P$ of shape $\lm$ and a cell $u\in\lm$.
  \item[Output] A pair $(Q,v)$ of a marked tableau $Q$ of shape $\lm$ and a cell $v\in\lm$.
    \begin{description}
    \item[Step 1] Set $Q=P$ and $v=u$. We call $v$ the \emph{active cell}.
    \item[Step 2] Let $(i,j)$ be the coordinate of the active cell $v$. Let $a_r=Q(i,j)$, $b_s=Q(i,j+1)$, and $c_t=Q(i+1,j)$. If $(i,j+1)\not\in\lm$ (resp.~$(i+1,j)\not\in\lm$), then set 
$b_s=(-1)_0$ (resp.~$c_t=(-1)_0$). If $a\ge b$ and $a>c$, then stop the process and return $(Q,v)$ as the output. Otherwise, there are two cases.
  \begin{itemize}
  \item If $b-1>c$, then set $Q(i,j)=(b-1)_s$ and $Q(i,j+1)=a_r$ as shown below, where the active cell $v$ is the cell containing $a_r$. Set $v=(i,j+1)$ and repeat Step~2. 
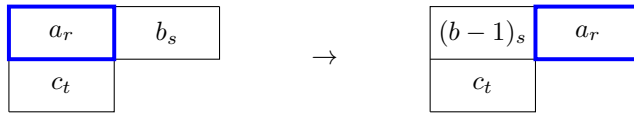
\begin{center}
\begin{tikzpicture}[scale=.7]
\dcell11{$a_r$} \dcell12{$b_s$} \dcell21{$c_t$}
\draw (0,-2)--(0,0)--(4,0);
\draw[color=blue, line width=1.5pt] (0,-1) rectangle (2,0);
\node at (6,-1) {$\rightarrow$};
\begin{scope}[shift={(8,0)}]
\dcell11{$(b-1)_s$} \dcell12{$a_r$} \dcell21{$c_t$}
\draw (0,-2)--(0,0)--(4,0);
\draw[color=blue, line width=1.5pt] (2,-1) rectangle (4,0);
\end{scope}
\end{tikzpicture} 
\end{center}
  \item If $c+1\ge b$, then set $Q(i,j)=(c+1)_t$ and $Q(i+1,j)=a_r$ as shown below, the active cell $v$ is the cell containing $a_r$. Set $v=(i+1,j)$ and repeat Step~2. 
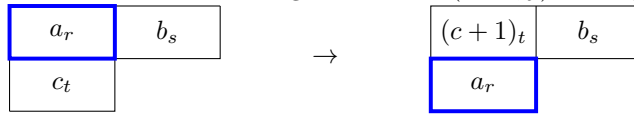
\begin{center}
\begin{tikzpicture}[scale=.7]
\dcell11{$a_r$} \dcell12{$b_s$} \dcell21{$c_t$}
\draw (0,-2)--(0,0)--(4,0);
\draw[color=blue, line width=1.5pt] (0,-1) rectangle (2,0);
\node at (6,-1) {$\rightarrow$};
\begin{scope}[shift={(8,0)}]
\dcell11{$(c+1)_t$} \dcell12{$b_s$} \dcell21{$a_r$}
\draw (0,-2)--(0,0)--(4,0);
\draw[color=blue, line width=1.5pt] (0,-2) rectangle (2,-1);
\end{scope}
\end{tikzpicture} 
\end{center}
  \end{itemize}
    \end{description}
  \end{description}
\end{algorithm}

See Figure~\ref{fig:jdt} for an example of the value-jdt algorithm. 

\begin{figure}
  \centering
\begin{tikzpicture}[scale=.6]
\cell13{$0_5$} \cell14{$3_3$} \cell15{$2_3$} \cell16{$4_2$} \cell17{$7_1$} 
\cell22{$7_2$} \cell23{$1_2$} \cell24{$6_0$} \cell25{$6_1$} \cell26{$5_0$} \cell27{$4_0$} 
\cell31{$2_\infty$} \cell32{$5_2$} \cell33{$6_1$} \cell34{$5_1$} \cell35{$5_0$} 
\cell36{$3_1$} \cell37{$3_1$} 
\cell41{$2_3$} \cell42{$1_2$} \cell43{$4_0$} \cell44{$4_1$} \cell45{$4_0$} \cell46{$2_1$} 
\cell51{$4_0$} \cell52{$3_1$} \cell53{$2_1$} \cell54{$1_1$} \cell55{$1_0$} \cell56{$1_2$} 
\draw (0,-5)--(0,-2)--(1,-2)--(1,-1)--(2,-1)--(2,0)--(7,0);
\draw[line width=1.5] (2,-1)--(3,-1)--(3,-3)--(5,-3)--(5,-5)--(4,-5)--(4,-4)--(2,-4)--(2,-1);
\draw[color=blue, line width=1] (2.5,-1.5) circle (.5);
\node at (-1,-2.5) {$P=$};
\begin{scope}[shift={(11,0)}]
\cell13{$0_5$} \cell14{$3_3$} \cell15{$2_3$} \cell16{$4_2$} \cell17{$7_1$} 
\cell22{$7_2$} \cell23{$7_1$} \cell24{$6_0$} \cell25{$6_1$} \cell26{$5_0$} \cell27{$4_0$} 
\cell31{$2_\infty$} \cell32{$5_2$} \cell33{$5_0$} \cell34{$5_1$} \cell35{$5_0$} 
\cell36{$3_1$} \cell37{$3_1$} 
\cell41{$2_3$} \cell42{$1_2$} \cell43{$3_1$} \cell44{$3_0$} \cell45{$2_0$} \cell46{$2_1$} 
\cell51{$4_0$} \cell52{$3_1$} \cell53{$2_1$} \cell54{$1_1$} \cell55{$1_2$} \cell56{$1_2$} 
\node at (-1,-2.5) {$Q=$};
\draw[line width=1.5] (2,-1)--(3,-1)--(3,-3)--(5,-3)--(5,-5)--(4,-5)--(4,-4)--(2,-4)--(2,-1);
\draw (0,-5)--(0,-2)--(1,-2)--(1,-1)--(2,-1)--(2,0)--(7,0);
\draw[color=red, line width=1] (4.5,-4.5) circle (.5);
\end{scope}
\end{tikzpicture} 
  \caption{If $u=(2,3)$ and $v=(5,5)$, then $\vjdt(P,u)=(Q,v)$ and $\mjdt(Q,v)=(P,u)$. In each diagram the positions that the active cell visits are enclosed by the thick polygon.}
  \label{fig:jdt}
\end{figure}
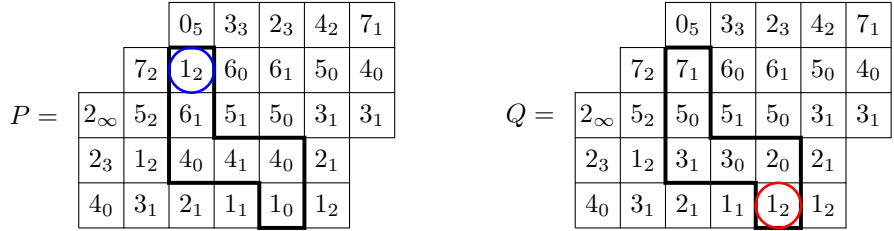

\begin{algorithm}[Mark-jeu de taquin]\label{alg:mjdt}
The \emph{mark-jdt algorithm} is described as follows. 
  \begin{description}
  \item[Notation] $\mjdt(Q,v)=(P,u)$.
  \item[Input] A pair $(Q,v)$ of a marked tableau $Q$ of shape $\lm$ and a cell $v \in\lm$.
  \item[Output] A pair $(P,u)$ of a marked tableau $P$ of shape $\lm$ and a cell $u \in\lm$.
    \begin{description}
    \item[Step 1] Set $P=Q$ and $u=v$. We call $u$ the \emph{active cell}. 
    \item[Step 2] Let $(i,j)$ be the coordinate of the active cell $u$. Let $a_r=P(i,j)$, $b_s=P(i,j-1)$, and $c_t=P(i-1,j)$. If $(i,j-1)\not\in\lm$ (resp.~$(i-1,j)\not\in\lm$), then set 
$b_s=\infty_\infty$ (resp.~$c_t=\infty_\infty$). If $r\le s$ and $r\le t$, then stop the process and return $(P,u)$ as the output. Otherwise, there are two cases.
  \begin{itemize}
  \item If $t<r\le s$, or $s,t<r$ and $b\ge c-1$, then set $P(i,j)=(c-1)_t$ and $P(i-1,j)=a_r$ as shown below, where the active cell $u$ is the cell containing $a_r$. Set $u=(i-1,j)$ and repeat Step~2. 
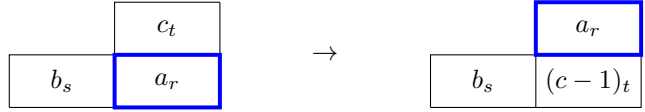
\begin{center}
\begin{tikzpicture}[scale=.7]
\dcell12{$c_t$} \dcell22{$a_r$} \dcell21{$b_s$}
\draw (0,-2)--(0,-1)--(2,-1)--(2,0)--(4,0);
\draw[color=blue, line width=1.5pt] (2,-2) rectangle (4,-1);
\node at (6,-1) {$\rightarrow$};
\begin{scope}[shift={(8,0)}]
\dcell12{$a_r$} \dcell22{$(c-1)_t$} \dcell21{$b_s$}
\draw (0,-2)--(0,-1)--(2,-1)--(2,0)--(4,0);
\draw[color=blue, line width=1.5pt] (2,-1) rectangle (4,0);
\end{scope}
\end{tikzpicture}
\end{center}

  \item If $s<r\le t$, or $s,t<r$ and $c>b+1$, then set $P(i,j)=(b+1)_s$ and $P(i,j-1)=a_r$ as shown below, where the active cell $u$ is the cell containing $a_r$. Set $u=(i,j-1)$ and repeat Step~2. 
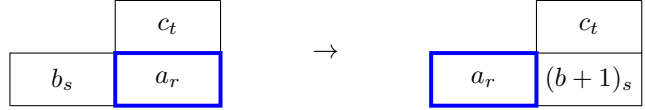
\begin{center}
\begin{tikzpicture}[scale=.7]
\dcell12{$c_t$} \dcell22{$a_r$} \dcell21{$b_s$}
\draw (0,-2)--(0,-1)--(2,-1)--(2,0)--(4,0);
\draw[color=blue, line width=1.5pt] (2,-2) rectangle (4,-1);
\node at (6,-1) {$\rightarrow$};
\begin{scope}[shift={(8,0)}]
\dcell12{$c_t$} \dcell22{$(b+1)_s$} \dcell21{$a_r$}
\draw (0,-2)--(0,-1)--(2,-1)--(2,0)--(4,0);
\draw[color=blue, line width=1.5pt] (0,-2) rectangle (2,-1);
\end{scope}
\end{tikzpicture} 
\end{center}
  \end{itemize}
    \end{description}
  \end{description}
\end{algorithm}

See Figure~\ref{fig:jdt} for an example of the value-jdt algorithm. 

Let $\lambda$ be a partition. An \emph{outer corner} of $\lambda$ is a cell $u\not\in\lambda$ such that $\lambda\cup\{u\}$ is a partition.  An \emph{inner corner} of $\lambda$ is a cell $u\in\lambda$ such that $\lambda\setminus\{u\}$ is a partition. For a skew shape $\lm$, a \emph{northwest corner} of $\lm$ is a cell in $\lm$ that is an outer corner of $\mu$ and a \emph{southeast corner} of $\lm$ is a cell in $\lm$ that is an inner corner of $\lambda$.
See Figure~\ref{fig:corner} for an example. 

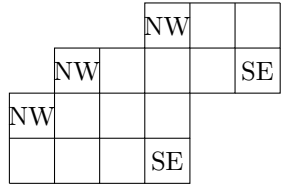
\begin{figure}
  \centering
\begin{tikzpicture}[scale=.6]
\cell14{NW} \cell15{} \cell16{}
\cell22{NW} \cell23{} \cell24{} \cell25{} \cell26{SE}
\cell31{NW} \cell32{} \cell33{} \cell34{}
\cell41{} \cell42{} \cell43{} \cell44{SE}
\draw (0,-4)--(0,-2)--(1,-2)--(1,-1)--(3,-1)--(3,0)--(6,0);
\end{tikzpicture}
  \caption{The northwest corners are the cells with an ``NW'' and the southeast corners are the cells with an ``SE''.}
  \label{fig:corner}
\end{figure}

\begin{defn}\label{def:vs}
Let $\alpha$ be a skew shape and $L\in\LHT^*_n(\alpha)$. Suppose that $r$ is the smallest mark and $a$ is the smallest value with mark $r$ in $L$. 
Then the \emph{tail} of $L$, denoted $\tail(L)$, is defined to be the rightmost cell $(i,j)\in\alpha$ with $L(i,j)=a_r$. See Figure~\ref{fig:LHT*} for an example.
\end{defn}

Note that for distinct cells $(i,j),(i',j')\in\lm$, if $L(i,j)=L(i',j')=a_r$, then  the fact that $L$ is an element in $\LHT^*_n(\lm)$ ensures that $j\ne j'$. Thus the tail of $L\in\LHT^*_n(\lm)$ is well-defined. It is clear from the definition that
the tail of $L$ is a southeast corner of $\lm$. 

\begin{defn}\label{def:ms}
Let $\beta$ be a skew shape and $S\in\SSCT^*_n(\beta)$. Suppose that $r$ is the largest mark
and $a$ is the largest value with mark $r$ in $S$. 
Then the \emph{head} of $S$, denoted $\head(S)$, is defined to be the leftmost cell $(i,j)\in\beta$ with $S(i,j)=a_r$. See Figure~\ref{fig:SSCT*} for an example.
\end{defn}

By a similar argument as before, one can check that if $S\in\SSCT^*_n(\beta)$, then $\head(S)$ is well-defined. Note, however, that $\head(S)$ is not necessarily a (northwest or southeast) corner of $\beta$.

We are now ready to define a map sending an extended $n$-lecture hall tableau $L\in \LHT^*_n(\lm)$ to a marked semistandard $n$-content tableau $S\in \SSCT^*_n(\lm)$.  Recall from the definition that in $L$ the marks are weakly decreasing along each row and column but the values are not sorted.  In $S$, on the contrary, the values are weakly decreasing along each row and strictly decreasing along each column but the marks are not sorted.
Our approach is, therefore, to sort the values of $L$ in order to obtain $S$, and to sort the marks of $S$ in order to obtain $L$. 
The two sorting algorithms are described below. See Figure~\ref{fig:typical} for an illustration of a typical situation and Figure~\ref{fig:sort} for a concrete example of these algorithms.

\begin{figure}
  \centering
\begin{tikzpicture}[scale=.4]
\draw (0,-8)--(0,-3)--(1,-3)--(1,-2)--(3,-2)--(3,-1)--(4,-1)--(4,0)--(11,0)--(11,-4)--(9,-4)--(9,-7)--(8,-7)--(8,-8)--(0,-8);
\draw [line width=1.5] (0,-5)--(3,-5)--(3,-4)--(4,-4)--(4,-2)--(7,-2)--(7,0);
\draw[color=blue, line width=1] (3.5,-3.5) circle (.5);
\draw[color=red, line width=1] (6.5,-5.5) circle (.5);
\draw [dashed, ->, line width=1] (3.5,-3.5) -- (5.5,-3.5)--(5.5,-5.5)--(6.5,-5.5);
\node at (-2,-4) {$T_i=$};
\node at (5.5,-1) {$\alpha_i$};
\node at (9,-1) {$\beta_i$};
\node at (2,-3.5) {$u_{i-1}$};
\node at (7.5,-5.5) {$v_i$};
\end{tikzpicture} 
  \caption{A typical diagram with $T_i$, $\alpha_i$, $\beta_i$, $u_{i-1}$, and $v_i$. 
The border between $\alpha_i$ and $\beta_i$ is shown with a thick path. The blue circle represents $u_{i-1}$ and the red circle represents $v_i$.
The dashed path represents the movement of the active cell in the process of $\vjdt(T_{i-1},u_{i-1})=(T_i,v_i)$.}
  \label{fig:typical}
\end{figure}
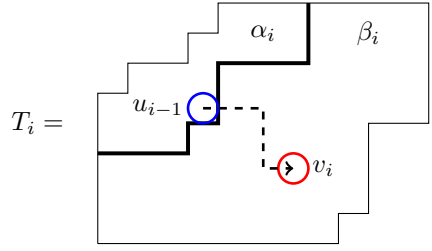

\begin{algorithm}[Value-sorting]\label{alg:vsort}
The \emph{value-sorting algorithm} is described as follows. 
  \begin{description}
  \item[Notation] $\vsort(L)=S$.
  \item[Input] An extended $n$-lecture hall tableau $L$ of shape $\lm$.
  \item[Output] A marked semistandard $n$-content tableau $S$ of shape $\lm$.
    \begin{description}
    \item[Step 1] Set $T_0=L$, $\alpha_0=\lm$, $\beta_0=\emptyset$, and $u_0=\tail(T_0)$. 
    \item[Step 2] For $i=1,2,\dots,|\lm|$, define $\alpha_i$, $\beta_i$, $T_i$, $u_i$, and $v_i$ recursively by 
      \begin{align*}
(T_i,v_i)&=\vjdt(T_{i-1},u_{i-1}),\\
\alpha_i & = \alpha_{i-1}\setminus\{u_{i-1}\},\\
\beta_i & = \beta_{i-1}\cup\{u_{i-1}\},\\
u_i &= \tail(T_i|_{\alpha_i}).
      \end{align*}
    \item[Step 3] Return  $S=T_{|\lm|}$ as the output. 
    \end{description}
  \end{description}
\end{algorithm}

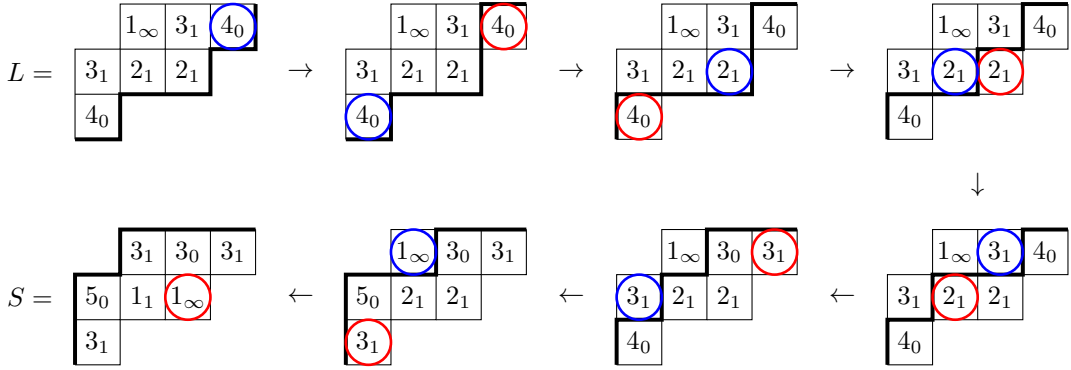
\begin{figure}
  \centering
\begin{tikzpicture}[scale=.6]
\node at (-1,-1.5) {$L=$};
\cell12{$1_\infty$} \cell13{$3_1$} \cell14{$4_0$}
\cell21{$3_1$} \cell22{$2_1$} \cell23{$2_1$}
\cell31{$4_0$} 
\draw (0,-3)--(0,-1)--(1,-1)--(1,0)--(4,0);
\draw[line width=1.5] (0,-3)--(1,-3)--(1,-2)--(2,-2)--(3,-2)--(3,-1)--(4,-1)--(4,0);
\draw[color=blue, line width=1] (3.5,-.5) circle (.5);
\node at (5,-1.5) {$\rightarrow$};
\begin{scope}[shift={(6,0)}]
\cell12{$1_\infty$} \cell13{$3_1$} \cell14{$4_0$}
\cell21{$3_1$} \cell22{$2_1$} \cell23{$2_1$}
\cell31{$4_0$} 
\draw (0,-3)--(0,-1)--(1,-1)--(1,0)--(4,0);
\draw[line width=1.5] (0,-3)--(1,-3)--(1,-2)--(2,-2)--(3,-2)--(3,-1)--(3,0)--(4,0);
\draw[color=red, line width=1] (3.5,-.5) circle (.5);
\draw[color=blue, line width=1] (.5,-2.5) circle (.5);
\node at (5,-1.5) {$\rightarrow$};
\end{scope}
\begin{scope}[shift={(12,0)}]
\cell12{$1_\infty$} \cell13{$3_1$} \cell14{$4_0$}
\cell21{$3_1$} \cell22{$2_1$} \cell23{$2_1$}
\cell31{$4_0$} 
\draw (0,-3)--(0,-1)--(1,-1)--(1,0)--(4,0);
\draw[line width=1.5] (0,-3)--(0,-2)--(1,-2)--(2,-2)--(3,-2)--(3,-1)--(3,0)--(4,0);
\draw[color=blue, line width=1] (2.5,-1.5) circle (.5);
\draw[color=red, line width=1] (.5,-2.5) circle (.5);
\node at (5,-1.5) {$\rightarrow$};
\end{scope}
\begin{scope}[shift={(18,0)}]
\cell12{$1_\infty$} \cell13{$3_1$} \cell14{$4_0$}
\cell21{$3_1$} \cell22{$2_1$} \cell23{$2_1$}
\cell31{$4_0$} 
\draw (0,-3)--(0,-1)--(1,-1)--(1,0)--(4,0);
\draw[line width=1.5] (0,-3)--(0,-2)--(1,-2)--(2,-2)--(2,-1)--(3,-1)--(3,0)--(4,0);
\draw[color=red, line width=1] (2.5,-1.5) circle (.5);
\draw[color=blue, line width=1] (1.5,-1.5) circle (.5);
\node at (2,-4) {$\downarrow$};
\end{scope}
\begin{scope}[shift={(18,-5)}]
\cell12{$1_\infty$} \cell13{$3_1$} \cell14{$4_0$}
\cell21{$3_1$} \cell22{$2_1$} \cell23{$2_1$}
\cell31{$4_0$} 
\draw (0,-3)--(0,-1)--(1,-1)--(1,0)--(4,0);
\draw[line width=1.5] (0,-3)--(0,-2)--(1,-2)--(1,-1)--(2,-1)--(3,-1)--(3,0)--(4,0);
\draw[color=red, line width=1] (1.5,-1.5) circle (.5);
\draw[color=blue, line width=1] (2.5,-.5) circle (.5);
\node at (-1,-1.5) {$\leftarrow$};
\end{scope}
\begin{scope}[shift={(12,-5)}]
\cell12{$1_\infty$} \cell13{$3_0$} \cell14{$3_1$}
\cell21{$3_1$} \cell22{$2_1$} \cell23{$2_1$}
\cell31{$4_0$} 
\draw (0,-3)--(0,-1)--(1,-1)--(1,0)--(4,0);
\draw[line width=1.5] (0,-3)--(0,-2)--(1,-2)--(1,-1)--(2,-1)--(2,0)--(3,0)--(4,0);
\draw[color=red, line width=1] (3.5,-.5) circle (.5);
\draw[color=blue, line width=1] (.5,-1.5) circle (.5);
\node at (-1,-1.5) {$\leftarrow$};
\end{scope}
\begin{scope}[shift={(6,-5)}]
\cell12{$1_\infty$} \cell13{$3_0$} \cell14{$3_1$}
\cell21{$5_0$} \cell22{$2_1$} \cell23{$2_1$}
\cell31{$3_1$} 
\draw (0,-3)--(0,-1)--(1,-1)--(1,0)--(4,0);
\draw[line width=1.5] (0,-3)--(0,-2)--(0,-1)--(1,-1)--(2,-1)--(2,0)--(3,0)--(4,0);
\draw[color=red, line width=1] (.5,-2.5) circle (.5);
\draw[color=blue, line width=1] (1.5,-.5) circle (.5);
\node at (-1,-1.5) {$\leftarrow$};
\end{scope}
\begin{scope}[shift={(0,-5)}]
\cell12{$3_1$} \cell13{$3_0$} \cell14{$3_1$}
\cell21{$5_0$} \cell22{$1_1$} \cell23{$1_\infty$}
\cell31{$3_1$} 
\draw (0,-3)--(0,-1)--(1,-1)--(1,0)--(4,0);
\draw[line width=1.5] (0,-3)--(0,-2)--(0,-1)--(1,-1)--(1,0)--(2,0)--(3,0)--(4,0);
\draw[color=red, line width=1] (2.5,-1.5) circle (.5);
\node at (-1,-1.5) {$S=$};
\end{scope}
\end{tikzpicture} 
  \caption{The value-sorting algorithm applied to
$L\in\LHT^*_n(\lm)$ returns $S\in\SSCT^*_n(\lm)$, where $n=7$, $\lambda=(4,3,1)$ and $\mu=(1)$. 
The mark-sorting algorithm is the reverse process. Each diagram represents $T_i$. The border between $\alpha_i$ and $\beta_i$ is drawn by a thick path. The blue circle indicates $u_i=\tail(T_i|_{\alpha_i})$  and the red circle indicates $v_i=\head(T_i|_{\beta_i})$.}
  \label{fig:sort}
\end{figure}

\begin{algorithm}[Mark-sorting]\label{alg:msort}
The \emph{mark-sorting algorithm} is described as follows. 
  \begin{description}
  \item[Notation] $\msort(S)=L$.
  \item[Input] A marked semistandard $n$-content tableau $S$ of shape $\lm$. 
  \item[Output] An extended $n$-lecture hall tableau $L$ of shape $\lm$.
    \begin{description}
    \item[Step 1] Set $T_{|\lm|}=S$, $\alpha_{|\lm|}=\emptyset$, $\beta_{|\lm|}=\lm$, and $v_{|\lm|}=\head(T_{|\lm|})$. 
    \item[Step 2] For $i=|\lm|-1,|\lm|-2,\dots,0$, define $\alpha_i$, $\beta_i$, $T_i$, $u_i$, and $v_i$ recursively by 
      \begin{align*}
(T_i,u_i)&=\mjdt(T_{i+1},v_{i+1}),\\
\alpha_i & = \alpha_{i+1}\cup\{u_{i+1}\},\\
\beta_i & = \beta_{i+1}\setminus\{u_{i+1}\},\\
v_i &= \head(T_i|_{\beta_i}).
      \end{align*}
    \item[Step 3] Return  $L=T_{0}$ as the output. 
    \end{description}
  \end{description}
\end{algorithm}

In order to show that the above algorithms are inverse to each other, we need the following two lemmas.

\begin{lem}\label{lem:value}
  Let $L\in\LHT^*_n(\lm)$.  Suppose that $\alpha_i$, $\beta_i$, $T_i$, $u_i$, and $v_i$ are given as in Algorithm~\ref{alg:vsort}. Then,
for each $i=1,2,\dots,|\lm|$, the following properties hold.
\begin{enumerate}
\item $T_i|_{\alpha_i}\in\LHT^*_n(\alpha_i)$
and $T_i|_{\beta_i}\in\SSCT^*_n(\beta_i)$. 
In particular, $T_{|\lm|}\in\SSCT^*_n(\lm)$. 
\item $\head(T_i|_{\beta_i})=v_i$.
\item $\mjdt(T_i,v_i)=(T_{i-1},u_{i-1})$. 
\end{enumerate}
\end{lem}
\begin{proof}
(1): We prove this for $i=0,1,\dots,|\lm|$ by induction. 
Since $T_0|_{\alpha_0}=L$ and $T_0|_{\beta_0}=\emptyset$, it is true for $i=0$. Let $1\le i\le |\lm|$ and suppose that (1) is true for $i-1$. 
Since $T_{i-1}|_{\alpha_{i-1}}\in\LHT^*_n(\alpha_{i-1})$, we have that $u_{i-1} = \tail(T_{i-1}|_{\alpha_{i-1}})$ is a southeast corner of $\alpha_{i-1}$. Hence, $\alpha_i  = \alpha_{i-1}\setminus\{u_{i-1}\}$ and $\beta_i = \beta_{i-1}\cup\{u_{i-1}\}$ are skew shapes. When we compute $(T_i,v_i)=\vjdt(T_{i-1},u_{i-1})$, 
 the value-jdt algorithm does not modify the cells in $\alpha_{i}$, which implies that
$T_{i}|_{\alpha_{i}} =T_{i-1}|_{\alpha_{i}}=L|_{\alpha_{i}} \in\LHT^*_n(\alpha_i)$ and $\vjdt(T_{i-1}|_{\beta_i},u_{i-1})=(T_i|_{\beta_i},v_i)$.
It is not hard to check that in the process of $\vjdt(T_{i-1}|_{\beta_i},u_{i-1})$ to obtain $T_i|_{\beta_i}$, the values of the cells in $\beta_{i}$ are weakly decreasing in each row and strictly decreasing in each column with only possible exceptions between the active cell and the cell to the right of it and the cell below it. When the process stops these two possible exceptions are resolved and we obtain $T_i|_{\beta_i}\in\SSCT^*_n(\beta_i)$ as desired. 

(2): It is clear from the construction that if $r$ is the largest mark and $a$ is the largest value with mark $r$ in $T_i|_{\beta_i}$, then $T_i(v_i)=a_r$. If $v_i$ is the only cell in $\beta_i$ with this property, then we have $\head(T_i|_{\beta_i})=v_i$. Otherwise, we must show that $v_i$ is the leftmost cell with this property. To this end suppose that $T_{i-1}(u_{i-1})=T_{i}(u_{i})=a_r$, $u_{i-1}=(k,l)$, $u_i=(k',l')$, and $v_{i}=(p,q)$, $v_{i+1}=(p',q')$. Then it is sufficient to show that $q'<q$. 
Since $T_{i-1}|_{\alpha_{i-1}}\in\LHT^*_n(\alpha_{i-1})$ and $u_{i-1}=\head(T_{i-1}|_{\alpha_{i-1}})$, we have $k'\ge k$ and $j'<j$. 

Let $u_{i-1}=w_0,w_1,w_2,\dots,w_d=v_{i}$ be the sequence of positions of the active cell in the construction of $\vjdt(T_{i-1},u_{i-1})=(T_i,v_i)$. 
We claim that when we compute $\vjdt(T_i,v_i)$, the active cell never enters the position $w_t$ if $w_{t+1}$ is south of $w_t$, for $0\le t<d$. 

Suppose that the claim is false. Then we can find the smallest integer $m$ such that $w_m=(g,h)$, $w_{m+1}=(g+1,h)$ and the active cell enters $w_m$. Considering the relative positions of $u_{i-1}$ and $u_i$, one can check that the active cell must enter $w_m$ from the east. Now we focus on the restrictions of $T_{i-1}$, $T_i$, and $T_{i+1}$ to the cells $(g,h-1),(g,h),(g+1,h-1),(g+1,h)$ as in  Figure~\ref{fig:22}.
Let $T_{i-1}(g+1,h-1)=b_x$ and $T_{i-1}(g+1,h)=c_y$. Since $T_{i-1}|_{\beta_{i-1}}\in\SSCT^*_n(\beta_{i-1})$, we have $b\ge c$. Considering the positions of the active cell in the process of $\vjdt(T_{i-1},u_{i-1})$ and $\vjdt(T_{i},u_{i})$, we obtain that $T_i(g+1,h-1)=b_x$, $T_i(g,h)=(c+1)_y$, $T_{i+1}(g+1,h-1)=b_x$, and $T_{i+1}(g,h-1)=c_y$. Since $T_{i+1}|_{\beta_{i+1}}\in\SSCT^*_n(\beta_{i+1})$, we have $b<c$, which is a contradiction to the above fact that $b\ge c$. 
Therefore, the claim is true. 

\begin{figure}
  \centering
\begin{tikzpicture}[scale=.7]
\dcell11{$*$} \dcell12{$*$} \dcell22{$c_y$} \dcell21{$b_x$}
\draw (0,-2)--(0,0)--(4,0);
\begin{scope}[shift={(7,0)}]
\dcell11{$*$} \dcell12{$(c+1)_y$} \dcell22{$*$} \dcell21{$b_x$}
\draw (0,-2)--(0,0)--(4,0);
\end{scope}
\begin{scope}[shift={(14,0)}]
\dcell11{$c_y$} \dcell12{$*$} \dcell22{$*$} \dcell21{$b_x$}
\draw (0,-2)--(0,0)--(4,0);
\end{scope}
\end{tikzpicture} 
  \caption{The restrictions of $T_{i-1}$ (on the left),
$T_i$ (in the middle), and $T_{i+1}$ (on the right) to the cells $(g,h-1),(g,h),(g+1,h-1),(g+1,h)$. }
  \label{fig:22}
\end{figure}

By the above claim, if $q'\ge q$, then the active cell in the process of $\vjdt(T_i,u_i)$ 
must move from $(z,q-1)$ to $(z,q)$ for some $z\ge p$. Suppose that $z=p$.
Let $T_i(p+1,q-1)=c_t$. Since $T_i(p,q)=a_r$, the fact that the active cell moved
from $(p,q-1)$ to $(p,q)$ implies that $a-1>c$. However, this means that
when the active cell was in $(p,q-1)$, its value is at most the value of the cell to the right and greater than the value of the cell below, and the value-jdt algorithm must stop at this stage, which is a contradiction. Therefore, we must have $z>p$. In this case,
since $T_{i+1}|_{\beta_{i+1}}\in\SSCT^*_n(\beta_{i+1})$
and $v_{i+1}$ is strictly below and weakly to the right of $v_i$, we have that the value of $T_{i+1}(v_{i+1})$ is less than the value of $T_{i+1}(v_i)$, which is a contradiction. Therefore, we must have $q'<q$, which completes the proof of (2).

(3): By the fact that $T_{i-1}|_{\beta_{i-1}}\in\SSCT^*_n(\beta_{i-1})$ and $T_{i-1}|_{\beta_{i-1}}\in\SSCT^*_n(\beta_{i})$, it is clear that the reverse process of $\vjdt(T_{i-1},u_{i-1})$ is given by the mark-jdt algorithm. We only need to check that
the process of $\mjdt(T_i,v_i)$ stops when the active cell reaches the cell $u_{i-1}$. 
Let $r$ be the largest mark and $a$ the largest value with mark in $T_i|_{\beta_i}$. 
Since $v_i=\head(T_i|_{\beta_i})$, we have $T_i(v_i)=a_r$ and $v_i$ is the leftmost cell with this property. Therefore, the movement of the active cell in 
the process of $\mjdt(T_i,v_i)$ continues until the active cell reaches at a northwest corner of $\beta_i$, which is $u_i$. If the active cell is at $u_i$, then
the fact that the mark of every cell in $\alpha_i$ is at least $r$ implies that the process of $\mjdt(T_i,v_i)$ stops.
\end{proof}

\begin{lem}\label{lem:mark}
  Let $S\in\SSCT^*_n(\lm)$.  Suppose that $\alpha_i$, $\beta_i$, $T_i$, $u_i$, and $v_i$ are given as in Algorithm~\ref{alg:msort}. Then,
for each $i=0,1,2,\dots,|\lm|-1$, the following properties hold.
\begin{enumerate}
\item $T_i|_{\alpha_i}\in\LHT^*_n(\alpha_i)$
and $T_i|_{\beta_i}\in\SSCT^*_n(\beta_i)$. 
In particular, $T_{0}\in\LHT^*_n(\lm)$. 
\item $\tail(T_i|_{\alpha_i})= u_i$.
\item $\vjdt(T_i,u_i)=(T_{i+1},v_{i+1})$. 
\end{enumerate}
\end{lem}
\begin{proof}
This lemma can be proved by arguments similar to those in the proof of Lemma~\ref{lem:value}. We omit the proof.    
\end{proof}

We now give a bijective proof of Theorem~\ref{thm:main2}. 

\begin{thm}\label{thm:inverse}
The map 
\[
\vsort:\LHT^*_n(\lm)\to \SSCT^*_n(\lm)
\]
is a weight-preserving bijection whose inverse is
\[
\msort: \SSCT^*_n(\lm)\to \LHT^*_n(\lm).
\]
\end{thm}
\begin{proof}
Lemmas~\ref{lem:value} and \ref{lem:mark} imply that 
the two maps $\vsort$ and $\msort$ are inverses of each other. 
Suppose $\vsort(L)=S$. In the process of the value-sorting algorithm,
the marks and the values with mark $\infty$ are never changed.
Therefore $\wt^*(L)=\wt^*(S)$. 
\end{proof}

\begin{remark}
  The bijection allows us to generate a random bounded lecture hall tableau of a given partition shape using Krattenthaler's random generation of a semistandard Young tableau. It will be interesting to extend this random generation to skew shapes.
In \cite{BLHT} a different algorithm is established using a Markov chain on bounded lecture hall  tableaux and
coupling from the past.
\end{remark}

\section{A connection between SSCT and SYT}
\label{sec:conn-betw-ssct}

In this section we use the weight-preserving bijection
$\vsort:\LHT^*_n(\lm)\to \SSCT^*_n(\lm)$
and its inverse $\msort: \SSCT^*_n(\lm)\to \LHT^*_n(\lm)$
to find a connection between $|\SSCT_n(\lm)|$ and $|\SYT(\lm)|$. 

Recall the sets $\SYT(\lm)$, $\SSYT_n(\lm)$, $\LHT_n(\lm)$, and $\SSCT_n(\lm)$ defined in the introduction. We also need the following definitions.

A tableau $T$ of shape $\lm$ is called \emph{standard} if every integer $1\le i\le |\lm|$ appears exactly once in $T$.  The set of standard tableaux of shape $\lm$ is denoted by $\ST(\lm)$. An $n$-content tableau of shape $\lm$ is a tableau $T$ of shape $\lm$ such that $0\le T(i,j) < n-i+j$ for all $(i,j)\in\lm$.  The set of $n$-content tableaux of shape $\lm$ is denoted by $\CT_n(\lm)$.
A \emph{hook tabloid} of shape $\lambda$ is a map $H:\lambda\to \ZZ$ satisfying
$-\leg(i,j) \le H(i,j)\le \arm(i,j)$ for all $(i,j)\in\lambda$,
where $\leg(i,j)=\lambda_j'-i$ and $\arm(i,j)=\lambda_i-j$.
We denote by $\HT(\lambda)$ the set of hook tabloids of shape $\lambda$. 

Let us now consider the map $\vsort:\LHT^*_n(\lm)\to \SSCT^*_n(\lm)$
restricted to the following sets:
\begin{align*}
X_n(\lm)&=\{L\in\LHT^*_n(\lm): \wt^*(L)=x_1\cdots x_{|\lm|}\},\\
Y_n(\lm)&=\{T\in\SSCT^*_n(\lm): \wt^*(T)=x_1\cdots x_{|\lm|}\}. 
\end{align*}
Since $\vsort$ is a weight-preserving bijection, we obtain the induced bijection
\[
\vsort:X_n(\lm)\to Y_n(\lm).
\]

We can naturally identify $L\in X_n(\lm)$ with the pair $(A,R)$ of tableaux of shape $\lm$: if $L(i,j)=a_r$ then $A(i,j)=a$ and $R(i,j)=r$. Then by the condition on $L$, we have $A\in\CT_n(\lm)$ and $R\in\SYT(\lm)$. This allows us to identify $X_n(\lm)$ with $\CT_n(\lm)\times \SYT(\lm)$. Similarly, we can identify $Y_n(\lm)$ with $\SSCT_n(\lm)\times \ST(\lm)$. Using this identification we can consider $\vsort$ as a bijection between these sets:
\begin{equation}
  \label{eq:vsort}
\vsort: \CT_n(\lm)\times \SYT(\lm) \to \SSCT_n(\lm)\times \ST(\lm).  
\end{equation}
Therefore we obtain the following corollary, which is a restatement of Proposition~\ref{prop:KS}. 

\begin{cor}\label{cor:prob}
For any skew shape $\lm$, we have
\[
\frac{|\SSCT_n(\lm)|}{\prod_{x\in\lm}(n+c(x))} = \frac{|\SYT(\lm)|}{|\lm|!},  
\]
which means that the probability that a random $T\in\CT_n(\lm)$ is semistandard 
is equal to the probability that a random $T\in\ST(\lm)$ is a standard Young tableau. 
\end{cor}

It is possible to understand the probabilistic description in Corollary~\ref{cor:prob} using the map \eqref{eq:vsort}. To this end we note that
each element $(A,B)\in \CT_n(\lm)\times \SYT(\lm)$ is a fixed point of $\vsort$,
i.e., $\vsort(A,B)=(A,B)$, if and only if $A\in\SSCT_n(\lm)$. 
Similarly, each element $(A,B)\in \SSCT_n(\lm)\times \ST(\lm)$ is a fixed point
of the inverse map $\msort=\vsort^{-1}$ if and only if $B\in\SYT(\lm)$. 
The probability that a random $A\in\CT_n(\lm)$ is an element in $\SSCT_n(\lm)$ is clearly equal to the probability that a random pair $(A,B)\in \CT_n(\lm)\times \SYT(\lm)$ satisfies $A\in\SSCT_n(\lm)$. In other words, this is the probability that
a random pair $(A,B)\in \CT_n(\lm)\times \SYT(\lm)$ is a fixed point of $\vsort$. 
By the same argument, we obtain that
the probability that a random $B\in\ST(\lm)$ is an element of $\SYT(\lm)$
is equal to the probability that a random pair $(A,B)\in \SSCT_n(\lm)\times \ST(\lm)$
is a fixed point of the map $\vsort^{-1}$. Since $\vsort$ and $\vsort^{-1}$ are inverses of each other with the same set of fixed points, we obtain that the two probabilities that we consider are equal.

We now consider the map \eqref{eq:vsort} for the case $\mu=\emptyset$.
Since $\SSCT_n(\lambda)=\SSYT_n(\lambda)$, we have the following bijection: 
\begin{equation}
  \label{eq:vsort2}
\vsort: \CT_n(\lambda)\times \SYT(\lambda) \to \SSYT_n(\lambda)\times \ST(\lambda).  
\end{equation}
 Recall the two bijections due to Novelli--Pak--Stoyanovskii and Krattenthaler. 

\begin{thm}[Novelli-Pak-Stoyanovskii]
For any partition $\lambda$, there is a bijection 
\[
\phi_{NPS}: \ST(\lambda)\to \SYT(\lambda)\times \HT(\lambda).
\]
\end{thm}

\begin{thm}[Krattenthaler]
For any partition $\lambda$, there is a bijection 
\[
\phi_{K}: \CT_n(\lambda)\to \SSYT_n(\lambda)\times \HT(\lambda).
\]
\end{thm}

Note that $\phi_{NPS}$ naturally induces a bijection
\[
\phi_{NPS}:\SSYT_n(\lambda)\times \ST(\lambda) \to 
\SSYT_n(\lambda)\times\HT(\lambda)\times\SYT(\lambda)
\]
by fixing the first component. Similarly $\phi_K$  induces a bijection
\[
\phi_K:\CT_n(\lambda)\times \SYT(\lambda) \to 
\SSYT_n(\lambda)\times\HT(\lambda)\times\SYT(\lambda).
\]
Then the three maps $\vsort$, $\phi_K$, and $\phi_{NPS}$ are bijections
among three sets $\CT_n(\lambda)\times \SYT(\lambda)$, 
$\SSYT_n(\lambda)\times \ST(\lambda)$, and $\SSYT_n(\lambda)\times\HT(\lambda)\times\SYT(\lambda)$,
see Figure~\ref{fig:maps}. These maps are not directly related. It might be interesting to find any connection between these maps.

\begin{figure}
  \centering
  \begin{tikzpicture}
\node at (-3,1.5) {$\CT_n(\lambda)\times \SYT(\lambda)$};
\node at (3,1.5) {$\SSYT_n(\lambda)\times \ST(\lambda)$};
\node at (0,.1) {$\SSYT_n(\lambda)\times\HT(\lambda)\times\SYT(\lambda)$};    
\draw [<->] (-.9,1.5)--(.9,1.5); 
\node at (0,1.8) {$\vsort$};
\node at (-1,.8) {$\phi_K$};
\node at (1.2,.8) {$\phi_{NPS}$};
\draw [<->] (-1,1.2)--(-.2,.5); 
\draw [<->] (1,1.2)--(.2,.5); 
  \end{tikzpicture}
  \caption{Three maps between three objects.}
  \label{fig:maps}
\end{figure}
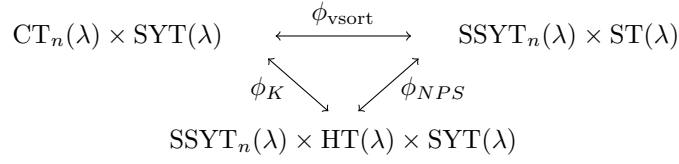

\section{Final remarks}
\label{sec:final}

Stanley \cite{Stanley72} showed that semistandard Young tableaux and standard Young tableaux fit together nicely in the framework of the $P$-partition theory, see also \cite[Chapter 3]{EC1} and \cite[Chapter 7]{Stanley}.
Lecture hall tableaux are also a special case of lecture hall $P$-partitions introduced by Br\"and\'en and Leander \cite{BL}. They found a connection between generating functions for the bounded lecture hall $P$-partitions and colored linear extensions of $P$. It will be interesting to compare our results with theirs.

 \begin{problem}
   Investigate bounded lecture hall tableaux using the results of Br\"and\'en and Leander \cite{BL}.
 \end{problem}

Krattenthaler's map \cite{Kratt1999} in fact gives a bijective proof of the following $q$-analog of \eqref{eq:7}, also due to Stanley \cite{Stanley1971}:
\[
\sum_{T\in\SSYT_n(\lambda)}q^{|T|}= q^{\sum_{i\ge1} (i-1)\lambda_i} \prod_{(i,j)\in\lambda} \frac{[n+c(i,j)]_q}{[h(i,j)]_q},
\]
where $[k]_q=1+q+q^2+\cdots+q^{k-1}$. If we only look at the values and ignore the marks, then our jeu de taquin slides in Algorithms~\ref{alg:vjdt} and \ref{alg:mjdt} are essentially the same as those in \cite{Kratt1999}. Recall that during these algorithms values are changing. Krattenthaler carefully designed his bijection so that these value changes are consistent with the value changes in hook tabloids. Our bijection, on the contrary, does not have hook tabloids, which makes it difficult to follow the change of values. If we can keep track of all the value changes, then it may be possible to find a refinement of Theorem~\ref{thm:main2}.
 
\begin{problem}
  Find a $q$-analogue of Theorem~\ref{thm:main2}. 
 \end{problem}

For a partition $\lambda=(\lambda_1,\dots,\lambda_\ell)$ with distinct parts, the \emph{shifted Young diagram} of
$\lambda$ is an array of squares in which the $i$th row has $\lambda_i$ squares and is shifted to the right by $i-1$ units. Standard Young tableaux and semistandard Young tableaux of a shifted shape can then be defined in a similar fashion. They also enjoy nice enumerative properties as in the case of a usual shape. 

\begin{problem}
Find a formula for the number of bounded $n$-lecture hall tableaux of a given shifted shape. 
\end{problem}

 Let $\delta_n=(n-1,n-2,\dots,1,0)$ and 
\[
d_{\lambda,\mu}^{(n)}=\det\left(\binom{\lambda_i+n-i}{\mu_j+n-j}\right)_{1\le i,j\le n},
\]
which is by Proposition~\ref{prop:KS} equal to $|\SSCT_n(\lm)|$. As mentioned in the introduction Lascoux \cite{Lascoux1978} used \eqref{eq:Lascoux}
to compute the Chern classes of the exterior square $\wedge^2 E$ and symmetric square $\Sym^2 E$
of a vector bundle $E$. To be more precise, let $c(E)=\prod_{i=1}^n(1+y_i)$ be the total Chern class of $E$. Lascoux showed that
\begin{align}
  \label{eq:ext}
c(\wedge^2E) &=\prod_{1\le i<j\le n}  (1+y_i+y_j) = \sum_{\mu\subset\delta_n} 2^{|\mu|-\binom n2} d^{(n)}_{\delta_n,\mu} s_\mu(y_1,\dots,y_n),\\
  \label{eq:sym}
c(\Sym^2E) &=\prod_{1\le i\le j\le n}  (1+y_i+y_j) = \sum_{\lambda\subset\delta_{n+1}} 2^{|\lambda|-\binom n2} d^{(n)}_{\delta_{n+1},\lambda} s_\lambda(y_1,\dots,y_n).
\end{align}

Billey, Rhoades, and Tewari \cite[Corollary~4.3]{BRT} found the following manifestly integral and positive formulas for the Schur expansions
of $c(\wedge^2E)$ and $c(\Sym^2E)$:
\begin{align}
  \label{eq:BRT1}
\prod_{1\le i<j\le n}  (1+y_i+y_j) &= \sum_{\mu\subset\delta_n} r^{(n)}_\mu s_\mu(y_1,\dots,y_n),\\
  \label{eq:BRT2}
\prod_{1\le i\le j\le n}  (1+y_i+y_j) &= \sum_{\lambda\subset\delta_{n+1}} \sum_{\substack{\mu\subseteq\lambda\cap\delta_n\\ \lm \mbox{ \small a vertical strip}}} 2^{|\lm|} r^{(n)}_\mu s_\lambda(y_1,\dots,y_n),
\end{align}
where a \emph{vertical strip} is a skew shape in which every row has at most one cell and
$r^{(n)}_\mu$ is the number of tableaux of shape $\mu$ such that
the entries are strictly decreasing along rows and weakly decreasing down columns, 
and every entry in row $i$ is in $\{1,2,\dots,n-i\}$. 

Comparing the Schur coefficients in \eqref{eq:ext}, \eqref{eq:sym}, \eqref{eq:BRT1}, and \eqref{eq:BRT2}, and using the fact $d^{(n)}_\lm=|\SSCT_n(\lm)|$, we obtain the following proposition.

\begin{prop}\label{prop:BRT}
For $\mu\subseteq\delta_n$ and $\lambda\subseteq\delta_{n+1}$, we have
\begin{align*}
|\SSCT_n(\delta_n/\mu)|&=2^{|\delta_n/\mu|}r^{(n)}_\mu,\\
|\SSCT_n(\delta_{n+1}/\lambda)|&=\sum_{\substack{\mu\subseteq\lambda\cap\delta_n\\ \lm \mbox{ \small a vertical strip}}}2^{|\delta_n/\mu|}r^{(n)}_\mu .
\end{align*}
\end{prop}

The objects in $\SSCT_n(\delta_n/\mu)$ and those counting $r_\mu^{(n)}$ have somewhat similar conditions on their entries but their shapes are complementary: $\delta_n/\mu$ and $\mu$. Understanding the connection between these two objects will be very interesting.

\begin{problem}
Find a bijective proof of  Proposition~\ref{prop:BRT}.
\end{problem}

 In a forthcoming paper \cite{BLHT}, the first author, Keating and Nicoletti show that  lecture hall tableaux
are in bijection with a certain dimer model on a graph whose faces are hexagons and octagons. Moreover they 
show that bounded lecture hall tableaux of a ``large" shape
exhibit the arctic curve phenomenon.

\end{document}